\documentclass[]{article}
\usepackage{amsmath}
\usepackage{amsfonts}
\usepackage{amssymb}
\usepackage{amsthm}
\usepackage{mathrsfs}
\usepackage{enumerate}
\usepackage{tikz}
\usepackage{MnSymbol,faktor, systeme, bbm} 
\usepackage{array}
\usepackage[all]{xy}
\usepackage{subfig}

\usepackage{hyperref}
\usepackage{pgf}

\DeclareMathOperator{\mdeg}{mdeg}
\DeclareMathOperator{\mmoddeg}{\tilde{m}deg}
\DeclareMathOperator{\Aut}{Aut}

\DeclareMathOperator{\content}{content}
\DeclareMathOperator{\Hom}{Hom}

\DeclareMathOperator{\Ann}{Ann}
\DeclareMathOperator{\Der}{Der}

\DeclareMathOperator{\Sym}{Sym}

\DeclareMathOperator{\ur}{ur}

\DeclareMathOperator{\leadingcoeff}{leadcoeff}

\DeclareMathOperator{\Rad}{Rad}

\newcommand{\intvars}{\integers[x_1,\ldots,x_{\ell}]}
\newcommand{\intDvars}{\integers_D[x_1,\ldots,x_{\ell}]}
\newcommand{\qvars}{\rationals[x_1,\ldots,x_{\ell}]}
\newcommand{\jrad}{\mathcal{J}}
\newcommand{\integers}{\ensuremath{\mathbb{Z}}}

\newcommand{\reals}{\ensuremath{\mathbb{R}}}
\newcommand{\C}{\ensuremath{\mathbb{C}}}
\newcommand{\rationals}{\ensuremath{\mathbb{Q}}}
\newcommand{\finitefield}{\ensuremath{\mathbb{F}_{p}}} %Really, this is the finite field with p elements.

\newcommand{\rd}{\ensuremath{\integers_D[x]}}
\newcommand{\md}{\ensuremath{M_D}}
\newcommand{\nd}{\ensuremath{N_D}}

\newcommand{\scrnd}{\ensuremath{N_D}}

\newcommand{\mq}{\ensuremath{\rationals M}}

\newcommand{\nq}{\ensuremath{\rationals N}}
\newcommand{\qr}{\ensuremath{\rationals[x]}}
\newcommand{\qm}{\ensuremath{\rationals M}}
\newcommand{\zdx}{\ensuremath{\integers_D[x]}}
\newcommand{\scrnp}{\ensuremath{N/pN}}
\newcommand{\Legendre}[2]{\ensuremath{ (#1/#2) }}

\newcommand{\freebycyclic}[2][d]{\ensuremath{\integers^{#1} \rtimes_{#2} \integers}}
\newcommand{\maxsubgr}[1][n]{\ensuremath{m_{#1}}}
\newcommand{\maxsubmod}[1][n]{\ensuremath{\tilde{m}_{#1}}}
\newcommand{\subgr}[1][n]{\ensuremath{a_{#1}}}
\newcommand{\mtriv}[1][n]{\ensuremath{\tilde{m}_{#1}^{\text{triv}}}}
\newcommand{\mnontriv}[1][n]{\ensuremath{\tilde{m}_{#1}^{\text{nontr}}}}

\newcommand{\submodisoto}[1]{\ensuremath{\tilde{m}_{#1}}}

\newcommand{\normal}{\ensuremath{\trianglelefteq}}
\newcommand{\ideal}{\ensuremath{\lhd}}
\newcommand{\maxideal}{\ideal_{\max}}
\newcommand{\tensor}{\ensuremath{\otimes}}

\newtheorem{proposition}{Proposition}%[chapter] %The part of this line that says '[chapter]' is the part that does the numbering by chapter.
\newtheorem{theorem}[proposition]{Theorem}
\newtheorem*{nntheorem}{Theorem}
\newtheorem{definition}[proposition]{Definition}
\newtheorem{corollary}[proposition]{Corollary}
\newtheorem{lemma}[proposition]{Lemma}

\newenvironment{case}[1]
{\vspace{5pt} \textbf{#1}}
{\vspace{5pt} \noindent \\}

\title{Maximal subgroup growth of some metabelian groups}
\author{Andrew James Kelley}

\begin{document}

\maketitle

\begin{abstract}
Let $\maxsubgr(G)$ denote the number of maximal subgroups of $G$ of index $n$. An upper bound is given for the degree of maximal subgroup growth
of all polycyclic metabelian groups $G$ (i.e., for $\limsup \frac{\log \maxsubgr(G)}{\log n}$, 
the degree of polynomial growth of $\maxsubgr(G)$). A condition is given for when this upper bound is attained.

%When specializing to groups of the form $G = \text{(f.g. abelian)} \rtimes \integers$, the growth rate of $\maxsubgr(G)$ is given more precisely. 
For $G = \integers^k \rtimes \integers$, where $A \in GL(k,\integers)$, it is shown that 
$\maxsubgr(G)$ grows like a polynomial of degree equal to the number of blocks in the rational 
canonical form of $A$. The leading term of this polynomial is the number of distinct roots (in $\C$) of the characteristic polynomial of the smallest block.
\end{abstract}

% Key words: Metabelian groups, maximal subgroup growth, maximal submodule growth

\section{Introduction}

Let $G$ be a f.g.\ (finitely generated) group, and let $\subgr(G)$ denote the number of subgroups of $G$ of index $n$. 
A highlight in subgroup growth is the theorem that gives an algebraic characterization of what it means for
the function $\subgr(G)$ to be bounded above by a polynomial in $n$, the so-called ``PSG Theorem'' (\textbf{p}olynomial
\textbf{s}ubgroup \textbf{g}rowth), which was proved by Lubotzky, Mann, and Segal. 
See \cite{Lubotzky2003} and the references there at the end of Chapter 5.

Much progress has been made in the area of subgroup growth, but there is no known general formula for 
calculating $\deg(G)$, the degree of polynomial growth of a given PSG (polynomial subgroup growth) group. 
In \cite{Shalev_On_the_degree} however, Shalev gives formulas for certain metabelian
groups and also for
all f.g.\ virtually abelian groups. Here,
\[
\deg(G) = \inf \{ \alpha | \subgr(G) \leq n^\alpha \text{ for all large $n$} \} = \limsup \frac{\log \subgr(G)}{\log n}.
\]

When it comes to \emph{maximal} subgroup growth, much progress also has been made. See for example
\cite{Mann1996}, where Mann relates polynomial maximal subgroup growth in profinite groups
to having a positive probability of topologically generating the group 
by picking a finite subset at random. See also the more recent \cite{Jaikin-Zapirain2011}, where
Jaikin-Zapirain and Pyber give a ``semi-structural characterization'' of polynomial maximal subgroup growth. 
However, just as in subgroup growth, there are only a few groups for which we know the exact degree of maximal subgroup growth.
It \emph{is} known for
free prosolvable groups of finite rank; this was determined by Lucchini, Menegazzo and Morigi
in \cite{Lucchini2006} together with Morigi's work in \cite{Morigi2006}.

Inspired by the progress Shalev made for calculating $\deg(G)$ in \cite{Shalev_On_the_degree}, I have worked on calculating the degree
of maximal subgroup growth. Notation:
\[
\maxsubgr(G) = \text{the number of maximal subgroups of $G$ of index $n$}
\]
\[
\mdeg(G) = \inf \{ \alpha | \maxsubgr(G) \leq n^\alpha \text{ for all large $n$} \} = \limsup \frac{\log \maxsubgr(G)}{\log n}
\]
How can we determine $\mdeg(G)$, for given $G$ in some nice class of groups?
How is $\mdeg(G)$ determined by the algebraic structure of $G$? This paper answers these question for
certain metabelian groups. 
%(In my dissertation, I have also calculated $\mdeg(G)$ for all f.g.\ virtually abelian groups and intend to rework that into
%a paper within a year or two.)

One of the two main results in this paper is the following theorem, which
gives an upper bound for $\mdeg(G)$ for all polycyclic metabelian groups. This theorem also gives a condition for when the upper bound is attained:
\begin{nntheorem}
	Let $G$ be a group with f.g.\ abelian normal subgroup $N$. Suppose $G/N$ is an abelian, $\ell_0$-generated group
	of torsion-free rank $\ell$. After choosing a generating set for $G/N$, $N$ becomes a $\integers[x_1,\ldots,x_{\ell_0}]$-module.
	Let $R = \integers[x_1,\ldots,x_{\ell_0}]$.
	Let $I = (x_1 - 1, x_2 - 1, \ldots, x_{\ell_0} - 1)_{R}$. Let $t$ be
	the torsion-free rank of (the abelian group) $N/IN$, and let $d = d_{\rationals \tensor_\integers R}(\rationals \tensor_\integers N)$
	(the minimal number of generators of $\rationals \tensor_\integers N$ as a $\rationals \tensor_\integers R$-module).
	Then
	\[
	  \mdeg(G) \leq \max \{\ell + t - 1, d \},
	\]
	with equality if both $G \cong N \rtimes G/N$ and $\ell \geq 1$.
\end{nntheorem}
\noindent This is Theorem~\ref{thm:(f.g. abelian) by (f.g. abelian)} below.
%
%The main result in this paper is Theorem~\ref{thm:(f.g. abelian) by (f.g. abelian)}, which
%gives an upper bound for $\mdeg(G)$ for all polycyclic metabelian groups (i.e. groups that are an
%extension of a f.g.\ abelian group by a f.g.\ abelian group). Also
%shown in this theorem is that the given upper bound is attained for all groups form
%\[
%\text{(f.g.\ abelian)} \rtimes A,
%\]
%where $A$ is a f.g.\ abelian group of torsion-free rank $\ell \geq 1$.

Of course, $\mdeg(G)$ is just an approximation of how fast $\maxsubgr(G)$ grows as $n \to \infty$. Sometimes, we can be more precise
than just giving $\mdeg(G)$. For f.g.\ groups of the form
\[
 G = \text{(arbitrary abelian)} \rtimes \integers,
\]
the growth type (see Definition \ref{def:growth type}) of $\maxsubgr(G)$ is given in Proposition \ref{prop:growth type of (arbitrary abelian) by Z}.

When we specialize to groups of the form
\[
G = \text{(f.g.\ abelian)} \rtimes \integers = N \rtimes \integers,
\]
we can be even more precise than giving the growth type of $\maxsubgr(G)$. Note that as $N$ is a normal subgroup of $G$,
$N$ becomes a $\integers[x]$-module. So $\rationals \tensor_\integers N$ is a f.g.\ module over the PID $\rationals[x]$. In this case,
we have the following theorem, the other main result of this paper: 
\begin{nntheorem}
	Let $G = N \rtimes \integers$, with $N$ f.g.\ as an abelian group. Let 
	\[
	\rationals \tensor_\integers N =  \bigoplus_{j=1}^d \rationals[x]/(a_j),
	\]
	where $a_1 | a_2 | \cdots | a_d$ as provided by the structure theorem of f.g.\ modules over PIDs (so with $a_1$ not a unit). 
	So $d = d_{\rationals[x]}(\rationals \tensor_\integers N)$.
	Also, let $\rho_1$ be the number of 
	(distinct) roots of $a_1$ in $\C$. Then
	\[\begin{aligned}
	\maxsubgr(G) & \leq \rho_1 n^{d} + O(n^{d-1}) & \text{for all large $n$, and}\\
	\maxsubgr(G) & \geq \rho_1n^{d} & \text{for infinitely many $n$.} 
	\end{aligned} 
	\]
\end{nntheorem}
This is
Theorem~\ref{thm:maximal subgroup growth of (f.g. abelian) by Z}.
The result stated in the second paragraph of the abstract is Corollary~\ref{cor:max subgr growth of Z^k semidir_A Z}.

The general method used here for finding the maximal subgroup growth of metabelian groups $N \rtimes A$
naturally falls into two parts:
\begin{itemize}
	\item find the maximal $\integers[A]$-submodules of $N$
	\item count derivations (1-cocycles) from $A$ to simple quotients of $N$
\end{itemize}
See Lemma~\ref{lem:abelian by something - counting maximal subgroups by derivations}.

The idea of reducing subgroup growth questions of metabelian groups to commutative algebra is not new. See Chapter 9
in \cite{Lubotzky2003}. Also, submodule growth has been considered
by Segal before in \cite{Segal_growth_of_ideals_and_submodules} and \cite{Segal_polynomial_ring}. Further, the use of 
derivations in counting subgroups is well established in subgroup growth; see the first page of Chapter 1 in 
\cite{Lubotzky2003} as well as Section 1.3.

%%%
Section \ref{sec:notation} gives notation (most but not all standard) which is used throughout the paper. 
%Section~\ref{sec:using derivations}
%shows how derivations can be used to count maximal subgroups, and Section~\ref{sec:counting derivations} shows how to count derivations
%in the relevant cases for this paper. Section~\ref{sec:Submodules counted by isomorphism type of quotient} gives important notation
%for counting submodules by their isomorphism type. Section~\ref{sec:Codimension 1 subspaces} views maximal submodules as codimension 1 subspaces.
%Section~\ref{sec:miscellaneous} contains several miscellaneous results (mostly known) that will be needed later.
Section~\ref{sec:preliminary results} shows how derivations can be counted and used for counting maximal subgroups in metabelian groups. It also
contains several miscellaneous results (mostly known) that will be needed later.
The goals of Section~\ref{sec:finitely generated Z[x]-modules} are to describe the maximal submodule growth of 
 (a) all $\integers_D[x]$-modules (with $D$ finite) which are finitely generated as $\integers_D$-modules
 and (b) all finitely generated  $\integers[x]$-modules. Section~\ref{sec:fin gen modules over polynomial ring with several variables} shows how to 
 count the maximal submodules of $\integers[x_1, x_2, \ldots, x_\ell]$-modules, which are finitely generated as abelian groups.
 Section~\ref{sec:certain metabelian groups} contains the main results of the paper, on the maximal subgroup growth of certain metabelian groups.
 It also works out the exact maximal subgroup growth of an example.

Finally, note that most of the work presented in this paper was done while I was a graduate student at Binghamton University and
is from \cite{kelley}, my dissertation. 

\subsection{Notation and Terminology}
\label{sec:notation}

\noindent $\subgr(G)$: the number of subgroups of $G$ of index $n$\\
$\maxsubgr(G)$:  the number of maximal subgroups of $G$ of index $n$\\
$\maxsubmod(N)$: the number of maximal submodules of $N$ of index $n$\\
$\submodisoto{S}(N)$: See Definition \ref{def:submodisoto - number of submodules with quotient iso to etc.}\\
%the number of quotients of $N$ isomorphic (as modules) to $S$\\
$\mtriv(N)$, $\mnontriv(N)$: See Definition \ref{def:mtriv and mnontriv}\\
%\qquad \qquad See Definition \ref{def:submodisoto - number of submodules with quotient iso to etc.}.
%Why can I not insert space here? I also tried \hspace{.2in}

\noindent $\Der(G,A)$: the set of derivations (see below) from $G$ to $A$ \\

%\noindent $H \leq_n G$ ($H \trianglelefteq_n G$): $H$ is a subgroup of $G$ of index $n$ (resp.\ and is normal)\\
\noindent $H \leq_n G$: $H$ is a subgroup of $G$ of index $n$\\
$H \leq_f G$ ($H \normal_f G$): $H$ is a subgroup of $G$ of finite index (resp.\ and is normal)\\
%$H \lneq G$: $H$ is a proper subgroup of $G$\\ %This is too obvious to include.
$I \maxideal R$: $I$ is a maximal ideal of $R$\\
$M \leq_{\max} N$: $M$ is a maximal submodule\footnote{Occasionally, the symbols `$\leq_{\max}$' will mean `maximal subgroup of'. Hopefully, the
	usage will be clear from context.} of $N$\\
$(a_1,\ldots,a_k)_R:$ the ideal of $R$ generated by $a_1,\ldots,a_k \in R$\\

%\noindent PMSG: polynomial maximal subgroup growth\\
\noindent $\mdeg(G)$: the degree of maximal subgroup growth\footnote{This is exactly what Mann
	denotes by $s^*(G)$ on page 449 of \cite{Mann1996}. Assuming
	$\maxsubgr(G) \geq 1$ for infinitely many $n$, then this also
	equals what Mann denotes by $s(G)$ on page 448 of that paper:
	$\limsup ((\log \maxsubgr(G))/\log n) =$ \newline $ \inf \{s | \maxsubgr(G) \leq Cn^s, \text{ for some C} \}$.
	
	Note that this differs from what Lubotzky defines on page 2 of \cite{Lubotzky2002}
	as the ``\thinspace`polynomial degree' of the
	rate of growth of $\maxsubgr(G)$'':
	$\mathcal{M}(G) := \sup_{n \geq 2} ((\log \maxsubgr(G))/ \log n)$.} of a group G: 
\[
\mdeg(G) = \inf \{ \alpha | \maxsubgr(G) \leq n^\alpha \text{ for all large $n$} \}.
\]
$\mmoddeg(N)$: the degree of maximal submodule growth\footnote{Of course,
	this depends on the ring $R$ (which is implicit, given $N$). Hence, though the notation $\mmoddeg_R(N)$ would be appropriate,
	we will not use the subscript $R$, especially since it is understood from the context.} of an $R$-module $N$:
\[
\mmoddeg(N) = \inf \{ \alpha | \maxsubmod(N) \leq n^\alpha \text{ for all large $n$} \}.
\]

\noindent $N \cong_R M$: $N$ and $M$ are isomorphic as $R$-modules\\
$d_R(N)$: the minimal size of an $R$-module generating set for $N$\\
$\rationals N$: $\rationals \tensor_\integers N$\\
$\integers_D$: the localization of $\integers$ at the (finite) set of primes $D$\\

Suppose $G$ acts on the abelian group $A$ on the left. Recall that a \emph{derivation} (also called a 1-cocycle, or crossed homomorphism)
is a function $\delta:G \to A$ that satisfies\footnote{If $A$ is not assumed to be abelian, and if $G$ instead acts on the right,
	the condition changes to $\delta(gh) = \delta(g)^h \cdot \delta(h)$ for all $g, h \in G$.} 
$\delta(gh) = \delta(g) + g\cdot \delta(h)$ for all $g, h \in G$.

Almost all groups that appear in this document as groups (except $\rationals$ which is a field$\ldots$) 
will be finitely generated (f.g.). %We use `f.g.' as shorthand for `finitely generated'.

% Include notation for the degree of maximal submodule growth

% In a ring $R$, we denote by $(r)$ the ideal generated by an element $r$.
% $a|b$ means $a$ divides $b$ (whenever divisibility makes sense)

%Special notation used only in localized places:
%\pi(m): the product of the distinct prime divisors of $m \in \integers$ %Used only in the chapter on Baumslag-Solitar groups
%$\maxnrm^G_n(N)$
%The following is no longer used, right? $a_{n}^{\ntrianglelefteq G}(N)$
% X^{\pm1}
% $\epsilon$ The identity for $F_d$ the free group on $d$ letters
% \hat{y}_j is y_1y_2\cdots y_j

% := means ``is defined to be'' or ``which was defined to be''

%Maybe use big Oh notation.
In the following definition, the (increasing and eventually positive) function $g$ has domain a subset of 
the positive integers of the form $\{k, k+1, k+2, \ldots\}$.
\vspace{-.05in} 
\begin{definition}
	\label{def:growth type}
	Let $f \colon \{1, 2, 3,\ldots \} \longrightarrow \reals$. We say that $f$ has growth type\ldots 
	\begin{itemize}
		\item[] \ldots at most $g$, if $f(n) = O(g(n))$ (using ``Big O'' notation\footnote{This means that
			for some constant $C$, we have $|f(n)| \leq C g(n)$ for all large $n$.}).
		\item[] \ldots at least $g$, if $f(n) = \Omega(g(n))$: that is, there exists some constant $C > 0$ such that
		$Cg(n) \leq f(n)$ for \textbf{infinitely many} $n$.
		\item[] And if $f$ has growth type at most $g$ and at least $g$, we say it has growth type $g$.
	\end{itemize}
\end{definition}

Note that just like an analogous definition in \cite{Lubotzky2003} (Section 0.1), this notion of ``has growth type'' is not symmetric.

\section{Preliminary results}
\label{sec:preliminary results}

We begin with an easy observation:

\begin{lemma}
	\label{lem:quotient+the_rest} 
	Let $G$ be a finitely generated group with $N \normal G$. Then 
	\[
	\maxsubgr(G) = \maxsubgr(G/N) + \text{``the complement type''}
	\]
	where ``the complement type'' is the number of index $n$ maximal subgroups $M$ of $G$ with $MN = G$.
\end{lemma}
\begin{proof}
	Either $M$ contains $N$ or it does not. The former case is equivalent to 
	$MN = M$, and the latter is equivalent to $MN = G$.
\end{proof}

So how do we count ``the complement type''? It turns out that if $N$ is abelian and itself has a complement in $G$ (a subgroup
$K \leq G$ such that $N \cap K = \{1\}$ and $KN = G$) then the answer to this question 
(Lemma~\ref{lem:obtaining maximal submodules}) is particularly nice. 

We now recall that a
group $B$ acting on an abelian group $A$ gives us a $\integers[B]$-module structure for $A$. As such, $A$ is called
a $B$ module. And so for a group $G$ with
an abelian normal subgroup $N$, for any $N_0 \leq N$ with $N_0 \normal G$ we have that $G$ acts on $N/N_0$ by 
conjugation. But since $N$ is abelian, $G/N$ acts on $N/N_0$ by conjugation, and so $N/N_0$ is 
a $\integers[G/N]$-module. 

\begin{lemma}
	\label{lem:obtaining maximal submodules}
	Let $N$ be an abelian normal subgroup of $G$. Suppose $M$ is a proper subgroup of $G$ with $MN = G$. Then $M \leq_{\max} G$ iff $M \cap N$
	 is a maximal $\integers[G/N]$-submodule of $N$. Also, $[G: M] = [N : M\cap N]$.
\end{lemma}
\begin{proof}
	This is just Result 5.4.2 from \cite{Robinson} reworded. Let $N_0 = M \cap N$. Indeed, $N_0$ being a maximal $\integers[G/N]$-submodule of $N$
	precisely means that $N_0$ is maximal among all the proper subgroups of $N$ which are normal in $G$, and this means that
	$N/ N_0$ is a minimal normal subgroup of $G/ N_0$.
\end{proof}

Recall that of course a submodule $N_0$ of $N$ is maximal iff $N/N_0$ is a simple module.

Before continuing, we make another comment about group rings. If $A$ is a free abelian group of rank $\ell$, then the group ring
$\integers[A]$ is just the Laurent polynomials in $\ell$ variables
with integer coefficients: $\integers[x_1,x_1^{-1},x_2,x_2^{-1},\ldots,x_\ell,x_\ell^{-1}]$. 

\subsection{Using derivations}
\label{sec:using derivations}

The following is well known:
\begin{lemma}
	\label{lem:complements correspond to derivations}
	Suppose 
	\[
	N \hookrightarrow G \overset{\pi}{\twoheadrightarrow} G/N
	\]
	is exact and that $\sigma$ is a splitting of $\pi$. Then there is a one-to-one correspondence between
	the set of complements to $N$ and $\Der(G/N, N)$ where the action of $G/N$ on $N$ is defined
	by ${}^{\bar{g}}n := \sigma(\bar{g})n\sigma(\bar{g})^{-1}$.
\end{lemma}
For a proof, see for example Corollary 2.13 in \cite{kelley}.  
%TODO: get Brown's book from the Colorado College library
%Also, see the very readable description on pages 86 through 88 of \cite{Brown1982} on Low-Dimensional Cohomology.

%If we want to be even more general, we could ask how
%to build a group $G$ such that $N \normal G$ and $G/N \cong K$.

The idea of using derivations to count subgroups is well established. See \cite{Lubotzky2003}, pages 11, 15. 
Another reference is \cite{Shalev_On_the_degree}. In fact, the origin of this section was wondering what Lemma 2.1 (iii) in 
\cite{Shalev_On_the_degree} reduced to when counting
maximal subgroups; the analogous result here is 
Lemma \ref{lem:abelian by something - counting maximal subgroups by derivations}.

\begin{lemma} 
	\label{lem:abelian by something - counting maximal subgroups by derivations}
	Let $G$ be a f.g.\ group with $N \normal G$ and $N$ abelian. Then
	\[ \tag{*}
	\maxsubgr(G) \leq \maxsubgr(G/N) + \sum_{N_0} \lvert\Der(G/N, N/N_0)\rvert
	\]
	where the sum is taken over all $N_0$ such that $N_0 \normal G$, $N_0 \leq N$ and such that $N/N_0$ is a simple $\integers[G/N]$-module with $|N/N_0| = n$. 
	When we have $G \cong N \rtimes G/N$, then the inequality in (*) is an equality.
\end{lemma}
\begin{proof}
	%USE Lemma \ref{lem:obtaining maximal submodules} ?	
	 For the inequality,  by Lemma~\ref{lem:quotient+the_rest},
	 we only need to show that the number of maximal subgroups $M$ of $G$ such that $MN = G$ is bounded above by
	 $\sum_{N_0} \lvert\Der(G/N, N/N_0)\rvert$.
	 
	 Let $M \leq_{\max} G$ with $MN = G$ and $[G:M] = n$. Let $N_0 = M\cap N$. Then by Lemma~\ref{lem:obtaining maximal submodules},
	 $N/N_0$ is a simple $\integers[G/N]$-module with $|N/N_0| = n$. We have the exact sequence
	 \[\tag{*1}
	   N/N_0 \hookrightarrow G/N_0 \twoheadrightarrow G/N.
	 \]
	 By Lemma \ref{lem:complements correspond to derivations}, $M$ is counted by the term $\lvert\Der(G/N, N/N_0)\rvert$, and we have that
	 distinct $M_1, M_2 \leq_{\max} G$ with $M_i \cap G = N_0$ for $i = 1,2$, correspond to different derivations. This proves (*).
	 
	 Next, suppose $G = N \rtimes G/N$. Let $N_0$ be a maximal $\integers[G/N]$-submodule of $N$ with $|N/N_0| = n$. 
	Let $\mathcal{M}$ be the set of maximal subgroups $M$ of $G$ (of index $n$) that
	have $MN = G$ and $M \cap N = N_0$. By Lemma~\ref{lem:obtaining maximal submodules}, $\mathcal{M}$ 
	``is'' (or rather, corresponds to) the set of complements to $N/N_0$ in $G/N_0$. Because $G/N_0$ is just $N/N_0 \rtimes G/N$,
	the short exact sequence (*1) splits. Therefore, by Lemma~\ref{lem:complements correspond to derivations}, $\mathcal{M}$ has
	cardinality $\lvert\Der(G/N, N/N_0)\rvert$.
\end{proof}
%TODO: review   Lemma 2.1 (iii) in \cite{Shalev_On_the_degree}

\subsection{Counting derivations}
\label{sec:counting derivations}
In order to actually use derivations to count maximal subgroups, we need to be able to count derivations.
%Edit this to make it smoother.

We begin by stating a slightly weaker version of Lemma 2.5 from \cite{Shalev_On_the_degree}. 
In the lemma here, notice that $A$ is a module over the group ring $\integers[\langle x \rangle]$. 
\begin{lemma}
	\label{lem:Shalev's counting derivations out of cyclic groups}
	Suppose a cyclic group $\langle x \rangle$ acts on a finite abelian group $A$. Also, suppose
	\begin{enumerate}[(i)]
		\item $\langle x \rangle$ is the infinite cyclic group, or
		\item $x$ has order $k$ and $(1 + x + x^2 + \cdots + x^{k-1}) \cdot a = 0$ for all $a \in A$.
	\end{enumerate}
	Then $\lvert\Der(\langle x \rangle, A)\rvert = |A|$. 
\end{lemma}

Note: In Shalev's paper, instead of $A$,
he has an arbitrary finite group $F$. 
The main reason why the lemma is not stated in that generality here
is to use additive notation for $A$. Also, instead of the second point, the lemma could instead say
\[
\lvert\Der(\langle x \rangle, A)\rvert = \lvert\{ a \in A : (1 + x + x^2 + \cdots + x^{k-1}) \cdot a = 0 \}\rvert.
\]

At this point, we could state Lemma \ref{lem:abelian by infinite cyclic - exact counting} (and prove it in one line). Readers may 
want to read that before continuing this section.

While Lemma \ref{lem:Shalev's counting derivations out of cyclic groups} tells us how to count derivations
if the domain is a cyclic group, we will have need to count derivations when the domain is not cyclic. To do so, 
we prove that derivations factor through quotients, just 
as homomorphisms factor through quotients. 

Let $G$ be a group acting on the (abelian) group $A$. 
Suppose $N \normal G$ and that $N$ acts trivially\footnote{By this we mean that $g\cdot a = a$ for all $g \in N$ and $a \in A$.}
on $A$.
% Maybe include the word ``Recall'' when introducing the action of G/N on A. (I should probably mention this sooner,
%or maybe I already have.
Recall that this gives us an action of $G/N$ on $A$.
Further, suppose $N$ is normally generated by 
$\{a_1, \ldots, a_k\}$.%: we mean that $N$ is the smallest normal subgroup of 
%$G$ containing $\langle a_1, \ldots, a_k \rangle$.  
\begin{lemma}
	\label{lem:derivations factor through quotients}
	With the above notation, suppose $\delta \colon G \longrightarrow A$ is a derivation and 
	that $\delta(a_i) = 0$ for all $i$. Then
	\begin{enumerate}[(i)]
		\item $\delta(g) = 0$ for all $g \in N$, and therefore
		\item $\delta$ factors through $G/N$.
	\end{enumerate}
\end{lemma}
Notes: (1) This basically is exercise 4(a) in \cite{Brown1982} (pg.\ 90). 
(2) The hypothesis that $A$ is abelian is not needed, but it simplifies the notation slightly; further, 
in what follows, the lemma is applied only in the case that $A$ is abelian.
\begin{proof}
	We have that $N$ is generated (as a subgroup) by the set of all $ga_i g^{-1}$ such that $g \in G$ and 
	$1 \leq i \leq k$. It is immediate from the definition of derivations that to prove (i), we only need to 
	show $\delta(ga_i g^{-1}) = 0$ for all $g$ and $a_i$. So pick $g$ and $a_i$.
	
	We have (explanations following the equations)
	\begin{align}
	\delta(ga_1 g^{-1}) &= \delta(g) + g\delta(a_i g^{-1}) \\
	&= \delta(g) + g(\delta(a_i) + a_i\delta(g^{-1})) \\
	&= \delta(g) - ga_i g^{-1} \delta(g) \\
	&= \delta(g) - \delta(g) = 0
	\end{align}
	Equations (1) and (2) follow from the definition of derivation; for (1), we associated
	$ga_i g^{-1}$ as $g(a_i g^{-1})$. For equation (3), besides distributing $g$, 
	we are using the hypothesis that $\delta(a_i) = 0$ for all $i$, 
	and we are also using the general fact\footnote{This fact can be easily checked by applying the definition
		of derivation to $\delta(x x^{-1}) = 0$.} that $\delta(x^{-1}) = -x^{-1}\delta(x)$ (where here $x = g$).
	% And for (3), we are also using the fact that the map ``multiply on the left by ga_i'' commutes
	% with multiplying by -1, but since the action of G on A is by automorphisms, this does hold.
	%(An action of a group on an algebraic object A is just a group homomorphism from G into
	%the automorphism group of A.)
	Equation (4) follows from (3) by using the fact that $ga_i g^{-1} \in N$, and recalling that 
	$N$ acts trivially on $A$. And so combining (1) - (4) gives $\delta(ga_1 g^{-1}) = 0$, 
	which proves part (i) of this lemma.
	
	For part (ii), we first claim that since $\delta(N) = \{0\}$, we get a well-defined function 
	$\bar{\delta} \colon G/N \longrightarrow A$ via $\bar{\delta}(gN) = \delta(g)$. Indeed, 
	take $g \in G$ and $n \in N$. Then $\delta(gn) = \delta(g) + g\delta(n)$, but this equals
	$\delta(g)$, since $\delta(n) = 0$ by part (i). What remains to be shown is that the function $\bar{\delta}$
	is a derivation.
	
	Take $g, h \in G$. Then $\bar{\delta}(gNhN) = \bar{\delta}(ghN) = \delta(gh)$, but since $\delta$ is a derivation, 
	$\delta(gh)$ equals $\delta(g) + g\delta(h)$, which equals $\bar{\delta}(gN) + gN\delta(hN)$, since the coset $gN$ acts
	on $A$ the way $g$ acts on $A$. 
\end{proof}

We prove the universal property (of free groups) for derivations. This is analogous to homomorphisms. 
Let $F_d$ be the free
group on $X = \{x_1, x_2, \ldots, x_d\}$. ($F_d$ is abelian if and only if $d = 1$.) Suppose $F_d$ acts on $A$. Note 
that $A$ is assumed\footnote{The reasons for this are (a) to simplify notation slightly and (b)
	because the author intends to use it only in the case that $A$ is abelian.} to be an abelian group.

\begin{lemma}
	\label{lem:universal property -- derivations}
	With the above notation, any map $\delta \colon \{x_1, x_2, \ldots, x_d\} \longrightarrow A$ gives a unique derivation 
	$\delta \colon F_d \longrightarrow A$.
\end{lemma}
Note: This is exercise 3(a) in \cite{Brown1982} (pg.\ 90).
\begin{proof}
	Let $x \in X$. Define $\delta(x^{-1}) := -x^{-1}\delta(x)$. Next,  
	for $y_1, y_2,\ldots,y_k \in X^{\pm1}$, let $y = y_1y_2\cdots y_k$, and assume $y$ is a reduced word. We will 
	then define
	$\delta(y)$ to be $\delta(y_1) + \sum_{j=1}^{k-1} y_1\cdots y_j\delta(y_{j+1})$; written out, this says
	\[
	\delta(y_1\cdots y_k) := \delta(y_1) + y_1\delta(y_2) + \cdots + y_1y_2\cdots y_{k-1}\delta(y_k).
	\]
	Let $\epsilon$ denote the identity of $F_d$; so $\epsilon$ is the empty word. So far, we have defined $\delta(y)$
	for any $y$ except $\epsilon$. We define $\delta(\epsilon) := 0$. We now have a well-defined function 
	$\delta \colon F_d \longrightarrow A$, and it is straightforward to check that $\delta$ is indeed a derivation:
	
	Let $y, z \in F_d$. If $y$ or $z$ (or both) are the identity, then $\delta(yz) = \delta(y) + y\delta(z)$. So
	suppose that neither $y$ nor $z$ is $\epsilon$, the identity.

	\begin{case}{Case 1.}
		Suppose that $yz$ is a reduced word.
	\end{case}
	It is easy to see that $\delta(yz) = \delta(y) + y\delta(z)$; indeed, let 
	$y = y_1y_2\cdots y_k$ and $z = z_1z_2\cdots z_\ell$, where $y_1, \ldots, y_k, z_1, \ldots, z_\ell \in X^{\pm1}$.
	To simplify notation, for $j \in \{1, 2, \ldots, k\}$, let $\hat{y}_j$ denote $y_1y_2\cdots y_j$ and similarly
	for $\hat{z}_j$ if $j \in \{1, 2, \ldots, \ell \}$. (So $y = \hat{y}_k$ and $z = \hat{z}_\ell$.) Then 
	\begin{align*}
	\delta(yz) &= \delta(y_1\ldots y_kz_1\ldots z_\ell) \\
	&= \delta(y_1) + \cdots +\hat{y}_{k-1}\delta(y_k) + y\delta(z_1) + yz_1\delta(z_2) + 
	\cdots + y\hat{z}_{\ell - 1}\delta(z_\ell) \\
	&= \delta(y) + y(\delta(z_1) + z_1\delta(z_2) + \cdots + \hat{z}_{\ell-1}\delta(z_\ell)) \\
	&= \delta(y) + y\delta(z),
	\end{align*}
	and this is what we wanted to show, finishing this case.

	\begin{case}{Case 2.}
		There is cancellation in the product $yz$.  
	\end{case}
	To show this case, we use induction on the amount of cancellation. Our base case is the previous case, that
	there is no cancellation. Note that $y$ and $z$ are each, individually, assumed still to be reduced words.
	Suppose $y = ux$ and $z = x^{-1}w$ for some $x \in X^{\pm1}$ and $u, w \in F_d$. Assume that
	$\delta(uw) = \delta(u) + u\delta(w)$. So since $yz = uxx^{-1}w = uw$, by our inductive hypothesis, we need
	only show that $\delta(ux) + ux\delta(x^{-1}w) =  \delta(u) + u\delta(w)$. We have (explanations following)
	\begin{align*}
	\delta(ux) + ux\delta(x^{-1}w) &= \delta(u) + u\delta(x) + ux(\delta(x^{-1}) +x^{-1}\delta(w)) \\
	&= \delta(u) + u\delta(x) + ux(-x^{-1}\delta(x) + x^{-1}\delta(w)) \\
	&= \delta(u) + u\delta(x) -uxx^{-1}\delta(x) + uxx^{-1}\delta(w) \\
	&= \delta(u) + u\delta(x) - u\delta(x) + u\delta(w) \\
	&= \delta(u) + u\delta(w)
	\end{align*}
	The first equality is by Case 1 applied to the reduced words $ux$ and $x^{-1}w$. The second equality
	just uses our definition of $\delta(x^{-1})$. Besides distributing $ux$, the third
	equality follows since the action of $F_d$ on $A$ is, of course, by automorphisms, 
	and hence we may pull the -1 in front. This finishes Case 2 and the lemma.
\end{proof}

For the rest of this section, we write $\integers^\ell = \langle x_1,\ldots, x_\ell | [x_i,x_j] \text{for all } i, j \rangle$ for the free 
abelian group of rank $\ell$ (written multiplicatively).

\begin{lemma}
	\label{lem:characterization of derivations from free abelian groups}
	Let $S$ be a simple $\integers^\ell$-module. There is a one-to-one correspondence between the set $\Der(\integers^\ell,S)$ and the 
	set of functions $\delta\colon \{x_1,\ldots,x_\ell \} \longrightarrow S$ satisfying
	\[ \tag{*}
	(1-x_i)\delta(x_j) = (1-x_j)\delta(x_i) \text{\hspace{.1in} for all $i,j$.}
	\]
\end{lemma}
\begin{proof}
	Step 1. Let $\delta\colon \integers^\ell \longrightarrow S$ be a derivation. Fix $i,j$. Because $x_ix_j = x_jx_i$, we have $\delta(x_ix_j) = \delta(x_jx_i)$ Therefore,
	$\delta(x_i) + x_i\delta(x_j) = \delta(x_j) + x_j\delta(x_i)$. Rearranging and factoring yields (*).
	
	Step 2. Let $\delta\colon \{x_1,\ldots,x_\ell \} \longrightarrow S$ satisfy (*). By Lemma \ref{lem:universal property -- derivations}, we
	get a unique derivation $\delta\colon F_\ell \longrightarrow S$, where the action of $F_\ell$ on $S$ is the 
	induced action.
	Fix $i,j$. We claim that $\delta([x_i,x_j]) = 0$. Indeed,
	\[\begin{aligned}
	\delta(x_ix_jx_i^{-1}x_j^{-1}) &= \delta(x_i) + x_i\delta(x_j) - x_ix_jx_i^{-1}\delta(x_i) - x_ix_jx_i^{-1}x_j^{-1}\delta(x_j)\\
	&= \delta(x_i) + x_i\delta(x_j) - x_j\delta(x_i) - \delta(x_j),
	\end{aligned}\] 
	where last equality is by the induced action.\footnote{Indeed, $x_ix_jx_i^{-1} = x_ix_jx_i^{-1}x_j^{-1}x_j = [x_i,x_j]x_j$. We then twice use 
		the fact that $[x_i,x_j]$ acts trivially on $S$.} Notice that this last expression is 0 precisely because (*) holds. Therefore, 
	Lemma \ref{lem:derivations factor through quotients} gives us a derivation from $\integers^\ell$ to $S$.
	
	Because Steps 1 and 2 are inverses of each other, we are finished. 
\end{proof}

\begin{lemma}
	\label{lem:counting derivations from free abelian groups to simple modules}
	Let $S$ be a simple $\integers^\ell$-module. Then
	\[
	\lvert \Der(\integers^\ell,S) \rvert = 
	\begin{cases}
	|S|^\ell & \text{if the action is trivial} \\
	|S| & \text{otherwise.}
	\end{cases}
	\]
	%  \begin{enumerate}[(a)]
	%   \item $|\Der(\integers^\ell,S)| = |S|^\ell$ if the action is 
	%   trivial\footnote{By this we mean that $a\cdot s = s$ for all $a \in \integers^\ell$ and $s \in S$.} and
	%   
	%   \item $|\Der(\integers^\ell,S)| = |S|$ otherwise.
	%  \end{enumerate}
\end{lemma}
\begin{proof}
	If the action is trivial, then 
	$\Der(\integers^\ell,S) = \Hom(\integers^\ell,S)$. 
	
	Assume the action is not trivial, and let $x_i \in \{x_1,\ldots,x_\ell\}$ be a generator\footnote{The free abelian 
		group is still written multiplicatively.} of $\integers^\ell$ that acts non-trivially
	on $S$. Then the action of $(1-x_i)$ on $S$ is invertible.\footnote{Of course, $S$ is a module over the ring $R = \integers[x_1,\ldots,x_\ell]$. Since
		$S$ is a simple $R$ module, then $S$ really is a 1-dimensional vector space. In this case, the function $x_i\cdot$ is just multiplication by some
		(non-identity) element of the field.} %Maybe say all this before the corollary.
	Though there is no element $(1-x_i)^{-1}$ in $\integers[x_1,\ldots,x_\ell]$, for $s_0 \in S$, we denote by 
	$(1-x_i)^{-1}s_0$ the image of $s_0$ under the image of the inverse automorphism of $(1-x_i)\cdot \in \Aut(S)$.
	
	Fix $j \neq i$. The equation $(1-x_i)\delta(x_j) = (1-x_j)\delta(x_i)$ is equivalent to the equation
	$\delta(x_j) = (1-x_i)^{-1}(1-x_j)\delta(x_i)$. Hence, by Lemma \ref{lem:characterization of derivations from free abelian groups}
	we may pick a derivation simply by picking $\delta(x_i)$ to be any
	element of $S$ and then defining $\delta(x_j)$ to be $(1-x_i)^{-1}(1-x_j)\delta(x_i)$.
\end{proof}

Our next goal is Lemma \ref{lem:counting derivations from f.g. abelian groups to simple modules}, which extends 
Lemma~\ref{lem:counting derivations from free abelian groups to simple modules} to the case that
the domain is any f.g.\ abelian group.

\begin{lemma}
	\label{lem:derivation vanishes on x^n}
	Let $S$ be a simple $\integers^\ell$-module. Assume that the action is non-trivial. Let $x \in \integers^\ell$ be such that the 
	automorphism $x \cdot \in \Aut(S)$ has finite order dividing some integer $n$. Let $\delta\colon \integers^\ell \longrightarrow S$ be a 
	derivation. Then
	\[
	\delta(x^n) = 0.
	\]
	% 
	%  Let the infinite cyclic group $\langle x \rangle$ act on $S$ non-trivially. Suppose that $S$ is a simple $\integers[\langle x \rangle]$-module,
	%  and that the order of $x \cdot \in \Aut(S)$ divides the integer $n$. Let $\delta\colon \langle x \rangle \longrightarrow S$ be a 
	%  derivation. Then
	%  \[
	%   \delta(x^n) = 0.
	%  \]
\end{lemma}
\begin{proof}
	Let $y \in \integers^\ell$ be such that $y \cdot \in \Aut(S)$ is non-trivial. 
	We know (similarly to Step 1 of Lemma \ref{lem:characterization of derivations from free abelian groups}) that 
	\[
	(1 - y)\delta(x^n) = (1 - x^n)\delta(y).
	\]
	But because $S$ is a simple module and $y\cdot$ is non-trivial, we get\footnote{Just like $(1 - x_i)$ in the proof of 
		Lemma~\ref{lem:counting derivations from free abelian groups to simple modules}\ldots} that the endomorphism $(1-y)\cdot$ is invertible. Therefore,
	$\delta(x^n) = (1-y)^{-1}(1 - x^n)\delta(y)$, but since the automorphism $x\cdot$ has order dividing $n$, we have that $x^n \cdot$ is the identity
	function on $S$. Therefore, $\delta(x^n) = (1-y)^{-1}(1 - x^n)\delta(y) = 0$.
\end{proof}

%If $G$ is any group and $A$ is a $G$ module for which the action is trivial, then it is follows immediately from the definition
%of derivations that $\Der(G,A) = \Hom(G,A)$. 
The following is a generalization of 
Lemma~\ref{lem:counting derivations from free abelian groups to simple modules}.
\begin{lemma}
	\label{lem:counting derivations from f.g. abelian groups to simple modules}
	Let $H$ be a f.g.\ abelian group. Let $S$ be a simple $H$-module. Then
	\[
	\lvert \Der(H,S) \rvert =
	\begin{cases}
	\lvert \Hom(H,S)\rvert & \text{if the action is trivial} \\
	|S| & \text{otherwise.}
	\end{cases}
	\]
\end{lemma}
\begin{proof}
	If the action is trivial, then $\Der(H,S) = \Hom(H,S)$. So suppose the action is non-trivial.
	
	Let $H$ be $\ell$-generated, and let $G = \integers^\ell$, the free abelian group of rank $\ell$. Let the action of $G$ on 
	$S$ be the induced action. By Lemma~\ref{lem:counting derivations from free abelian groups to simple modules}, we know that
	$\lvert \Der(G,S)\rvert = |S|$. %Let $\phi\colon G \longrightarrow H$ be a chosen surjection. 
	To prove this lemma, it is sufficient to show that each derivation from $G$ to $S$ gives a derivation (via
	Lemma \ref{lem:derivations factor through quotients}) from $H$ to $S$.\footnote{The following is clear: Let $\delta_1$
		and $\delta_2$ be different derivations from $G$ to $A$ that satisfy the hypotheses of 
		Lemma~\ref{lem:derivations factor through quotients} (for some given $N \normal G$). Then the lemma produces different
		derivations from $G/N$ to $A$.}
	
	Let $\delta \in \Der(G,S)$. Let $\pi\colon G \longrightarrow H$ be a surjection with kernel $N$. Let $x \in G$ 
	be such that $\pi(x)$ has order $n$. %Technically, I did not assume that $H$ itself was not free abelian. But if it was, then this lemma was proved above.
	(So $x^n$ is an arbitrary element of $N$.) In order to apply 
	Lemma \ref{lem:derivations factor through quotients}, it is sufficient to show that $\delta(x^n) = 0$ (for any such $x^n$). 
	We have $x\cdot \in \Aut(S)$ has order dividing $n$, since $N$ acts trivially on $S$. Thus $\delta(x^n) = 0$ by
	Lemma \ref{lem:derivation vanishes on x^n}.
\end{proof}

% We conclude with an easy lemma that may not be used in the rest of this document: A derivation followed by a homomorphism gives a derivation.
% 
% \begin{lemma}
%  Let $\phi\colon G \longrightarrow H$ be a group homomorphism and $\delta\colon H \longrightarrow A$ be a derivation. Then
%  the composition $\delta \circ \phi : G \longrightarrow A$ is a derivation.\footnote{Of course, $H$ has some action on $A$, and the action
%  of $G$ on $A$ is given by $g\cdot a := \phi(g) \cdot a$.}
% \end{lemma}
% \begin{proof}
%  This is straightforward. Let $g, h \in G$. Then
%  \[\begin{aligned}
%    \delta(\phi(gh)) &  = \delta(\phi(g)\phi(h)) \\
%                     & = \delta(\phi(g)) + \phi(g)\delta(\phi(h)) \\
%                     & = \delta(\phi(g)) + g\delta(\phi(h)).
%    \end{aligned}\]
% \end{proof}

\subsection{Submodules counted by isomorphism type of quotient}
\label{sec:Submodules counted by isomorphism type of quotient}

Let $R$ be ring, and let $N$ be an $R$-module. It is well known that for every maximal submodule $M$ of $N$, we have 
$N/M \cong_R R/I$ for some maximal left ideal $I \ideal R$.\footnote{If $R$ is commutative, then $I$ is the annihilator of $N/M$. If $R$ is
	not necessarily commutative, then we may take any element $a \in N/M$, with $a \neq 0$. Let $I$ be the kernel of the map $r \mapsto ra$. Since $N/M$ is
	simple, the map is surjective (because it is nonzero). We conclude $N/M \cong_R R/I$.}

In order to organize all the maximal submodules
of $N$ of a given index by the $R$-module isomorphism type of the quotient, we give the following definition:
\begin{definition}
	\label{def:submodisoto - number of submodules with quotient iso to etc.} 
	Let $S$ be a (finite) simple $R$-module. Then $$\submodisoto{S}(N)$$ denotes the number of submodules $M$ of $N$ such that
	$N/M \cong_R S$.
\end{definition}
We now state the following lemma:
\begin{lemma}
	\label{lem:maxsubmod N = sum_S submodisoto_S(N)}
	Let $N$ be a f.g.\ $R$-module. Then
	\[
	\maxsubmod(N) = \sum_S \submodisoto{S}(N),
	\]
	where the sum is taken over all simple $R$-modules of cardinality $n$. If $R$ is commutative, then also
	\[
	\maxsubmod(N) = \sum_I \submodisoto{R/I}(N)
	\]
	where the sum is taken over all maximal ideals $I$ of $R$ that have $|R/I| = n$.
\end{lemma}
\begin{proof}
	The first equality holds because we can partition the set of maximal submodules by the $R$-module isomorphism type of their quotient. 
	
	The second equality then follows by the well-known fact mentioned in the first paragraph of this section together with one other
	well-known fact: 
	Because $R$ is now assumed to be commutative, if we have two maximal ideals 
	$I_1, I_2 \ideal R$ with $I_1 \neq I_2$ but $|R/I_1| = |R/I_2|$ finite,  then $R/I_1$ and $R/I_2$ are not isomorphic 
	$R$-modules,\footnote{This is because their annihilators (namely $I_1$ and $I_2$ respectively) are different.}
	(even though they are 
	isomorphic fields). %\footnote{This is because for any $p^k$, there is only one field (up to isomorphism) of order $p^k$.} 
\end{proof}

Note: Recall that if $R$ is not commutative, then it is possible for $R/I_1 \cong_R R/I_2$ as $R$-modules, even if $I_1 \neq I_2$.\footnote{For example, let
	$R = M_2(\finitefield)$. Let $I_1 = \left( \begin{smallmatrix} 0& * \\ 0& *\end{smallmatrix} \right)$ 
	and $I_2 = \left( \begin{smallmatrix} *& 0 \\ *&0 \end{smallmatrix} \right)$. Then $R/I_1$ and $R/I_2$ are both isomorphic to the unique (up to iso.) simple 
	$R$-module.}

\subsection{Codimension 1 subspaces}
\label{sec:Codimension 1 subspaces}
Let $R$ be a commutative (unital) ring, and % (In the case that $R$ is not commutative, we are not quite dealing with vector spaces though.)
let $I \ideal R$ be maximal with $|R/I| = n$.

\begin{lemma}
	\label{lem:R/I of N to R/I of N/IN}
	With the notation from Definition \ref{def:submodisoto - number of submodules with quotient iso to etc.},
	 $$\maxsubmod[R/I](N) =  \maxsubmod[R/I](N/IN).$$
\end{lemma}
\begin{proof}
	It is immediate that $\maxsubmod[R/I](N) \geq  \maxsubmod[R/I](N/IN)$. Let $M$ be a maximal submodule of $N$ with $N/M \cong_R R/I$.
	We have $\Ann_R(N/M) = \Ann_R(R/I) = I$. Thus $IN$ is 0 mod $M$, i.e., $IN \subseteq M$. Therefore $\maxsubmod[R/I](N) \leq  \maxsubmod[R/I](N/IN)$.
\end{proof}

The following is very well known. 

\begin{lemma}
	\label{lem:N/IN is iso to  R/I tensor N}
	With the above notation,
	\[
	R/I \tensor_R N \cong_R N/IN.
	\]
\end{lemma}

\begin{lemma}
	\label{lem:maxsubmod iso to R/I and dimension of tensor product}
	Recall that $n = |R/I|$. We have
	\[
	\maxsubmod[R/I](N) = 1 + n + n^2 + \cdots + n^{s-1},
	\]
	where $s = \dim_{R/I}(R/I \tensor_R N)$.
\end{lemma}
\begin{proof}
	Lemma \ref{lem:R/I of N to R/I of N/IN} gives $\maxsubmod[R/I](N) =  \maxsubmod[R/I](N/IN)$, which itself is equal %I'd write
	% "which is equal to" but that causes spacing problems.
	to  \newline $\maxsubmod[R/I](R/I \tensor N)$ by 
	Lemma \ref{lem:N/IN is iso to  R/I tensor N}. Note that $R/I \tensor N$ is an $R/I$-vector space, and that its maximal submodules
	are codimension 1 subspaces, the number of which is the number of dimension 1 subspaces. Thus
	\[
	\maxsubmod[R/I](R/I \tensor_R N) = \frac{n^s - 1}{n - 1},
	\]
	where $s = \dim_{R/I}(R/I \tensor_R N)$ as desired.
\end{proof}

We get the following consequence of Lemma \ref{lem:maxsubmod iso to R/I and dimension of tensor product}:

\begin{corollary}
	\label{cor:submodisoto R/I for direct sum of cyclic modules}
	Recall $I \maxideal R$, with $|R/I| = n$. Suppose $N_1,\ldots, N_r$ are cyclic $R$-modules, and let 
	$s = |\{N_i: \submodisoto{R/I}(N_i) = 1 \}|$. Then
	\[
	\submodisoto{R/I}(N_1 \oplus N_2 \oplus \cdots \oplus N_r) = 1 + n + n^2 + \cdots + n^{s-1}.
	\]
\end{corollary}

\subsection{Miscellaneous}
\label{sec:miscellaneous}
We collect here a few more results (almost all well known) that we will use later.

How does passing to quotients affect the maximal subgroup growth? The following lemma shows that if we mod out by a finite subgroup,
then the maximal subgroup growth remains unchanged. (The question was inspired by Lemma 2.3 from \cite{Shalev_On_the_degree}.)
\begin{lemma}
	\label{lem:mod out by finite subgroup -- maximal subgroup growth is unchanged}
	Let $G$ be a f.g.\ group and $F \normal G$ finite. Let $n \in \integers_{\geq 1}$. If $n > |F|$, then
	\[
	\maxsubgr(G) = \maxsubgr(G/F).
	\]
\end{lemma}
\begin{proof}
	We will show that if a maximal subgroup does not contain $F$, then it has
	index at most $|F|$. Let $M \leq_n G$ be maximal and suppose
	that $F \nsubseteq M$. Since $F \normal G$, we get that $FM$ is a subgroup of $G$.
	Since $FM$ properly contains $M$, we conclude that $FM = G$. Therefore,
	\[
	[G : M] = [FM: M] = [F : F \cap M] \leq |F|.
	\]
	%We have just shown that if a maximal subgroup does not contain $F$, then its index is at most $|F|$.
\end{proof}

A similar statement works for maximal submodule growth. Let $R$ be a (unital) ring.
\begin{lemma}
	\label{lem:mod out by finite submodule -- maximal submodule growth is unchanged}
	Let $N$ be an $R$-module and $F \leq N$ a finite submodule. Let $n \in \integers_{\geq 1}$. If $n > |F|$, then
	\[
	\maxsubmod(N) = \maxsubmod(N/F).
	\]
\end{lemma}
\begin{proof}
	This is similar to our proof of Lemma \ref{lem:mod out by finite subgroup -- maximal subgroup growth is unchanged}. Let $M \leq_n N$ be a maximal.
	Suppose $F \nsubseteq M$. Then $n \leq |F|$ because
	\[
	N/M = (M + F)/M \cong_R F/M\cap F.
	\]
\end{proof}

---------------------------------------\\
The following will be used without comment throughout this document. For a proof, see for example Result 5.4.3 (iii) in \cite{Robinson}.
\begin{lemma}
	Let $G$ be a solvable group, and let $M$ be a maximal subgroup of $G$ of finite index. Then
	$[G:M]$ is a power of a prime.
\end{lemma}

---------------------------------------\\

Let $S$ be a $G$ module. Following \cite{Dummit2004} (page 798), we will denote by $S^G$ the set of all elements of $S$ that are fixed
by $G$: $S^G = \{s \in S: gs = s \text{ for all } g \in G \}$. If $S^G \neq \emptyset$, we say that $S$ has a fixed point. We now make an easy observation:

\begin{lemma}
	\label{lem:simple module - one fixed point implies trivial and has order a prime}
	Let $S$ be a simple (finite) $G$ module that has a fixed point. Then $S = S^G$ and $|S|$ is prime.
\end{lemma}
\begin{proof}
	The set $S^G$ is a submodule of $S$. Since it is non-empty and $S$ is simple, we get $S = S^G$. Since the action is trivial, a simple
	$G$ module is the same thing as a simple abelian group.
\end{proof}

%Have I already used the following somewhere? Is it too obvious to state by itself?
%\begin{lemma}
% Let $M \leq_n N$ be a maximal submodule of $N$ of index $n$. If the action of $R$ on $N/M$ is trivial, then $n$ is prime.
%\end{lemma}
---------------------------------------\\

Our next goal is the well-known Lemma \ref{lem:two full lattices - each scales into the other}. 
We first prove the main part of that lemma.

\begin{lemma}
	\label{lem:vector v gets scaled into a full lattice}
	Let $D$ be an integral domain and $F$ its field of fractions. Fix $d \geq 1$. Suppose $A$
	is a $D$-submodule of $F^d$ that is isomorphic (as a $D$-module) to $D^d$. Let $v \in F^d$.
	Then there exists $c_0 \in D$ with $c_0 \neq 0$ such that $c_0 v \in A$.
\end{lemma}
\begin{proof}
	The case when $d = 1$ is clear.
	
	Let $X = \{x_1, \ldots, x_d\}$ be a $D$-module generating set for $A$. We claim that the $F$-span of $X$
	is $F^d$. By contradiction, suppose that $X$ is linearly dependent over $F$. So there exist
	$a_1, \ldots, a_d \in F$ (not all zero) such that 
	\[
	a_1 x_1 + \cdots + a_d x_d = 0.
	\]
	By clearing the denominators we get
	\[
	\tilde{a}_1 x_1 + \cdots + \tilde{a}_d x_d = 0
	\]
	for some $\tilde{a}_1, \ldots, \tilde{a}_d \in D$ (not all zero) a contradiction; this proves our claim.
	The claim tells us that there exist $\alpha_1, \ldots, \alpha_d \in F$ such that 
	\[
	\alpha_1 x_1 + \cdots + \alpha_d x_d = v.
	\]
	Again, clearing the denominators finishes the proof.	
\end{proof}

\begin{lemma}
	\label{lem:two full lattices - each scales into the other}
	%Let $D = \integers$  and $F = \rationals$. 
	Fix $d \geq 1$. Suppose $A$ and $B$
	are $\integers$-submodules of $\rationals^d$ both isomorphic (as $\integers$-modules) to $\integers^d$. Then there exists 
	$c \in \integers$ such that $cB \leq_f A$.
\end{lemma}
\begin{proof}
	Let $X = \{y_1, \ldots, y_d\}$ be a $\integers$-module generating set for $B$. We then apply 
	Lemma~\ref{lem:vector v gets scaled into a full lattice} to each $y_i$ to get nonzero constants
	$c_1, \ldots, c_d \in \integers$ such that $c_i y_i \in A$. Then $c = \Pi_1^d c_i$ works.
\end{proof}

\begin{corollary}
	\label{cor:two full lattices in Q^d - one contained in the other => the index is finite}
	Let $A$, $B$ be f.g.\ subgroups of $\rationals^d$ such that $B \leq A$ and that $\rationals B = \rationals^d$. Then
	$[A : B]$ is finite.
\end{corollary}
\begin{proof}
	This follows from Lemma~\ref{lem:two full lattices - each scales into the other}. Notice that as $\integers$-modules, 
	$A$ and $B$ are both isomorphic to $\integers^d$. We are done because $cB \leq_f A$ for some $c$, implies that $B \leq_f A$ too.
\end{proof}

---------------------------------------\\
If we 
start with a non-constant polynomial $f \in \integers[x]$, does $f$ split mod $p$ for infinitely many primes $p$? It turns out
that slightly more is true, as the following lemma states.

%I removed Lemma \ref{lem:non-constant polynomials in Z[x] split infinitely often} and have this instead:
\begin{lemma}
	\label{lem:roots of non-constant polynomials in F_p and in C}
	Let $f \in \mathbb{Z}[x]$ be a non-constant polynomial. Consider $\bar{f} \in \mathbb{F}_p[x].$ 
	Let $\rho_p$ be the number of distinct roots of $\bar{f}$ in $\mathbb{F}_p$, and let $\rho$ be the number of 
	distinct roots of $f$ in $\mathbb{C}$. Then $\rho_p = \rho$ for infinitely many primes $p$.
\end{lemma}

For a proof, see \cite{MathOverflow_Rivin},  the answer Igor Rivin gave at MathOverflow to the author's question. (Or see Keith Conrad's answer
to the same question.)

\begin{lemma}
	\label{lem:num_roots in overline(F_p) of polynomial bounded above by num_roots in C}
	Let $f \in \mathbb{Z}[x]$ be a non-constant polynomial. Consider $\bar{f} \in \mathbb{F}_p[x].$ 
	Let $\bar{\rho_p}$ be the number of distinct roots of $\bar{f}$ in $\overline{\mathbb{F}_p}$, and let $\rho$ be the number of 
	distinct roots of $f$ in $\mathbb{C}$. Then $\bar{\rho_p} \leq \rho$ for all large primes $p$.	
\end{lemma} %THE following is almost entirely a cut and paste of Eric Wofsey's argument on math.stackexchange
%\begin{proof}
%We can factor $f$ as $$f=af_1^{e_1}\dots f_n^{e_n}$$ where $a\in\mathbb{Z}$ and $f_1,\dots,f_n\in\mathbb{Z}[x]$ 
%are distinct irreducible polynomials of positive degree.  By Gauss's lemma, the $f_i$ are also irreducible in $\mathbb{Q}[x]$.  
%An irreducible polynomial over a field characteristic $0$ is separable, and distinct irreducible polynomials over a field have 
%no roots in common with each other in any extension field. (See for example 
%Corollary 34 on page 574 of \cite{Dummit2004}.) 
%So each $f_i$ contributes $\deg f_i$ distinct roots to $f$ over $\mathbb{C}$, 
%and so $f$ has $\sum_i \deg f_i$ distinct roots over $\mathbb{C}$.

%Now, let $p$ be such that $\bar{f}\neq 0$.  Reducing our factorization of $f$ mod $p$ we get a factorization 
%$$\bar{f}=\bar{a}\bar{f}_1^{e_1}\dots \bar{f}_n^{e_n}.$$  Since $\bar{f}\neq 0$, none of these factors are $0$.  
%But now it is clear that $\bar{f}$ cannot have more than $\sum_i \deg \bar{f}_i\leq\sum_i\deg f_i$ roots in $\overline{\mathbb{F}}_p$, 
%since each root of $\bar{f}$ is a root of some $\bar{f}_i$.  Thus $\bar{\rho_p}\leq \rho$.
%\end{proof}
For a proof, see the answer Eric Wofsey gave to the author's question at \url{https://math.stackexchange.com/q/2753743}.

---------------------------------------\\
\begin{definition}
	\label{def:degree of a function}
	Let $k \in \integers_{\geq 1}$. For a function $$f \colon \{ k, k+1, k+ 2, \ldots \} \to \reals_{\geq 0}$$
	which is bounded above by a polynomial, define
	\[
	\deg(f) := \inf \{\, \alpha \mid f(n) \leq n^\alpha \text{ for all large } n \}.
	\]
\end{definition}
Notes: (1) If $f$ itself is a polynomial, then this agrees with the normal use of the term ``degree''.
(2) We have that $\mdeg(G) = \deg(\maxsubgr(G))$.

\begin{lemma}
	\label{lem:degree of a sum is the max of the degrees}
	Let $k \in \integers_{\geq 1}$, and let
	\[
	f, g, h \colon \{ k, k+1, k+ 2, \ldots \} \to \reals_{\geq 0}
	\]
	each be bounded above by a polynomial. Then
	\[
	\deg(f + g) = \max\{\deg(f), \deg(g) \}, \text{\quad and}
	\]
	\[
	\deg(f + g + h) = \max\{\deg(f), \deg(g), \deg(h) \}.
	\]
\end{lemma}
\begin{proof}
	We prove the first equality, and then the second follows by applying the first equality twice.\footnote{We could use induction to prove a more general
		lemma about the sum of $n$ functions.}
	
	Certainly $\deg(f+g) \geq \deg(f)$ because $f(n) + g(n) \geq f(n)$ for all $n$. Similarly, $\deg(f+g) \geq \deg(g)$. Hence,
	$\deg(f+g) \geq  \max\{\deg(f), \deg(g) \}$.
	
	Let $\alpha := \deg(f)$ and $\beta := \deg(g)$. Let $\varepsilon > 0$. Then
	\[
	f(n) \leq n^{\alpha + \varepsilon/2} \text{\quad and \quad} g(n) \leq n^{\beta + \varepsilon/2} \text{ for all large } n.
	\]
	Thus for all large $n$,
	\[
	\begin{aligned}
	f(n) + g(n) &\leq n^{\alpha + \varepsilon/2} + n^{\beta + \varepsilon/2}\\
	&\leq 2n^{\max\{\alpha, \beta \} + \varepsilon/2 }\\
	&\leq n^{\max\{\alpha, \beta \} + \varepsilon},
	\end{aligned}
	\]
	where in the last inequality, $n$ is large enough such that $2 \leq n^{\varepsilon/2}$. The inequalities give us
	that $\deg(f+g) \leq \max\{\deg(f), \deg(g) \}$. Hence $\deg(f+g) = \max\{\deg(f), \deg(g) \}$.
\end{proof}

\section{Finitely generated $\integers[x]$-modules}
\label{sec:finitely generated Z[x]-modules}

The goals of this section are to describe the maximal submodule growth of 
\begin{itemize}
	\item all $\integers_D[x]$-modules (with $D$ finite) which are finitely generated as $\integers_D$-modules
	\item all finitely generated  $\integers[x]$-modules.
\end{itemize} %\footnote{Readers who feel uncomfortable about such modules may want to read the first part of section \ref{sec:using matrices}, ``Using matrices''.}
For the latter, the cyclic case is about finding maximal ideals in $R = \integers[x]$ and in quotients $R/I$ of $R$. The general
case is handled by looking at f.g.\ modules over $\finitefield[x]$ (or over $\rationals[x])$ and applying the
well known structure theorem for f.g.\ modules over principal ideal domains. At that point, we need only appeal to
\S \ref{sec:Submodules counted by isomorphism type of quotient}.
%Did I reference the correction section?
\subsection{Cyclic $\integers[x]$-modules}

Let $R = \integers[x]$. As is well-known, a cyclic $R$ module is just (isomorphic to) $R/I$ where $I$ is an ideal of $R$; $I$ would be the annihilator of
a chosen generator. 

We first review what the maximal ideals of $R$ are:

\begin{lemma}
	\label{lem:maximal ideals of integral polynomial ring}
	The maximal ideals of $R = \integers[x]$ are precisely the ideals of the form $(p, f)$ where $p$ is a prime number and $f \in R$ is a polynomial that is
	irreducible mod $p$.
\end{lemma}
\begin{proof}
	Though this is very well known, an argument is given here. (A reference is 
	example 3(d) in \cite{Dummit2004} in the section titled ``The prime spectrum of a ring''.) 
	
	Let $I \ideal R$ be maximal. Since $R$ itself is not a field, $I$ is not the zero ideal. So there is an $a \in I$ with $a \neq 0$. 
	We claim that $I$ contains a prime number. Indeed, if $a \in \integers$, then, then the characteristic of the field $R/I$ is finite and hence prime. 
	On the other hand, if $a$ is a non-constant polynomial, then $R/I$ is finitely generated as an abelian group, and in this case, $\rationals$ cannot be
	a subgroup of $R/I$. So in this case, we also know that the characteristic of $R/I$ is finite. Hence $I$ does contain a prime number.
	
	So since $R/I$ is a quotient of  $\finitefield[x]$, the lemma follows. 
\end{proof}

We next note that maximal ideals of $\integers[x,x^{-1}]$ correspond exactly with the maximal ideals of $\integers[x]$ 
except for $(p, x)$, which are not maximal in $\integers[x,x^{-1}]$ because $x$ is a unit there.
See for example, Proposition 38 in \cite{Dummit2004} in the section titled ``Localization''. We easily get the following observation:

\begin{lemma}
	\label{lem:laurant polynomial ring - maximal ideals like polynomial ring}
	We have 
	\[
	\maxsubmod(\integers[x,x^{-1}]) = \begin{cases}
	\maxsubmod(\integers[x]) - 1 & \text{ when $n$ is prime}\\
	\maxsubmod(\integers[x]) & \text{ when $n$ is not prime.}
	\end{cases} 
	\]
\end{lemma}

In the following well-known result, $\mu$ is the m\"obius function.
\begin{lemma}
	\label{lem:mobious function - exact number of irreducibles}
	We have 
	\[
	\maxsubmod[p^k](\finitefield[x]) = \frac{1}{k}\sum_{a|k} \mu\left(\frac{k}{a}\right)p^a.
	\]
\end{lemma}
For a proof, see for example the last two pages of the section titled ``Finite Fields'' in \cite{Dummit2004}.

\begin{corollary}
	\label{cor:mmoddeg of finitefield[x]}
	%Let $R$ be $\finitefield[x]$. % or $\finitefield[x,x^{-1}]$. 
	The growth type of $\maxsubmod(\finitefield[x])$ is $n/\log(n)$.
\end{corollary}
Note that $\mmoddeg(\finitefield[x]) = 1$ even though $\maxsubmod(\finitefield[x])$ grows sublinearly.

\begin{lemma}
	\label{lem:maximal ideal growth of Z x}
	We have $\maxsubmod(\integers[x])$ has growth type $n$. In fact, $\maxsubmod(\integers[x]) \leq n$ for all $n$
	and $\maxsubmod(\integers[x]) = n$ when $n$ is prime.
\end{lemma}
\begin{proof}
	We know already that $\maxsubmod[p](\integers[x]) = p$, and therefore $\maxsubmod(\integers[x])$ has at least linear growth. To show that it has at most linear growth, 
	we may appeal to Lemma \ref{lem:mobious function - exact number of irreducibles} or make the following simpler observation:
	
	The number of monic polynomials in $\finitefield[x]$ of degree $k$ is exactly $p^k$.
	But since $\maxsubmod[p^k](\finitefield[x])$ is the number of \emph{irreducible}, monic polynomials of degree $k$, 
	we conclude that $\maxsubmod(R) \leq n$ for all $n$.
\end{proof}

Let $R = \integers[x]$, and let $I \ideal R$, so that $R/I$ is an ``arbitrary'' cyclic $R$ module. Recall that the content of a polynomial in $R$
is the greatest common divisor of its coefficients.

\begin{lemma}
	\label{lem:bound for primes not dividing content}
	Let $f \in I$ be a non-constant polynomial. Then for all primes $p$ which do not divide $\content(f)$
	we have that for all $k$,
	$$\maxsubmod[p^k](R/I) \leq \deg(f).$$ 
\end{lemma}
\begin{proof}
	Let $p$ be a prime that does not divide $\content(f)$. Then $\bar{f}$ in $\finitefield[x]$ is not zero. 
	Note that $\maxsubmod[p^k](R/I) = \maxsubmod[p^k](R/(p,I))$. Since $R/(p,I)$ is
	a quotient of $\finitefield[x]/(\bar{f})$, we get
	that $\maxsubmod[p^k](R/(p,I)) \leq \maxsubmod[p^k](\finitefield[x]/(\bar{f}))$. Next, recall that $\finitefield[x]$ is a PID, and the maximal ideals
	of $\finitefield[x]/(\bar{f})$ are exactly the ideals of the form $(g)$, where $g$ is an irreducible factor of $\bar{f}$.
	Just note that $\bar{f}$ has at most  $\deg(\bar{f}) \leq \deg(f)$ irreducible factors. 
\end{proof}

\begin{lemma}
	\label{lem:upper bound for finite F_p x module}
	Fix a prime $p$. Let $J \ideal \finitefield[x]$, and let $g \in J$ be nonzero. Then for all $k \geq 1$, we have
	\begin{enumerate}[(a)]
		\item $\displaystyle{\maxsubmod[p^k](\finitefield[x]/J) \leq \left\lfloor \frac{\deg(g)}{k} \right\rfloor}$.
		\item $\displaystyle{\maxsubmod[p^k](\finitefield[x]/J) \leq r}$, where $r$ is the number of distinct roots of $g$ in $\overline{\finitefield}$.
	\end{enumerate}
\end{lemma}
\begin{proof}
	This is similar to the proof of Lemma \ref{lem:bound for primes not dividing content}. For (a), we simply note that $g$ has at most 
	$\left\lfloor \frac{\deg(g)}{k} \right\rfloor$ irreducible factors of degree $k$. (If $g$ is constant, then it has 0 irreducible factors.)
	
	For (b), notice that the number of distinct irreducible factors of $g$ is bounded above by $r$.
\end{proof}

Again, let $R = \integers[x]$, and let $I \ideal R$. 

\begin{lemma}
	\label{lem:cyclic modules over polynomial ring in 1-var with growth type n / log n}
	Let $I \neq \{0\}$. Suppose that for some prime $p$, we have that $I \subseteq (p)$. Then $\maxsubmod(R/I)$ has growth type $n/\log(n)$.
\end{lemma}
\begin{proof}
	We get that $\finitefield[x]$ is a quotient of  $R/I$. Therefore, by Corollary \ref{cor:mmoddeg of finitefield[x]}, the growth
	type of $R/I$ is at least $n/\log(n)$. We next just need to prove that the maximal submodule growth can be no larger; this uses
	the fact that $I$ must contain a nonzero element.
	
	%(This case may be omitted.) 
	%If $I$ contains a nonzero integer $a$, then $\maxsubmod[q^k](R/I)$ is 0 for all primes $q$ not dividing $a$
	%(because then, $(a,q) = R$ in which case $I = R$). The result then follows since
	%for each prime $q$ dividing $a$, we can appeal again to Corollary \ref{cor:mmoddeg of finitefield[x]} (with $q$ in place of $p$).
	
	It is easy to see that $I$ contains a non-zero polynomial; indeed, 
	let $0 \neq a \in I$, and let $0 \neq g(x) \in R$. Hence $a g(x) \in I$, since $I$ is an ideal of $R$. So let $f$ be any
	non-constant polynomial in $I$. 
	Then by Lemma \ref{lem:bound for primes not dividing content},
	we get $\maxsubmod(R/I) \leq \deg(f)$ for all $n$ that are powers of some prime $p$ that does not divide $\content(f)$.
	And to finish, for the primes $q$ which do divide $\content(f)$, we may yet again appeal to Corollary \ref{cor:mmoddeg of finitefield[x]}.
\end{proof}

\begin{lemma}
	\label{lem:cycclic modules over polynomial ring in 1-var with no growth - ie. bounded by a constant}
	Suppose that for every prime $p$, we have that $I \nsubseteq (p)$. Then there is a constant $c$ such that $\maxsubmod(R/I) \leq c$ for all $n$.
\end{lemma}
\begin{proof}
	Just as in the proof of Lemma \ref{lem:cyclic modules over polynomial ring in 1-var with growth type n / log n}
	(first sentence of second paragraph), we have that $I$ must contain a non-constant polynomial $f$ (because $I \neq \{0\}$,
	for otherwise $I \subseteq (p)$ for all primes $p$). For primes not dividing
	content($f$), just apply Lemma \ref{lem:bound for primes not dividing content}. And for primes dividing content$(f)$, just use other polynomials:
	
	Let $X = \{p_1, p_2, \ldots, p_t\}$ be the primes dividing content$(f)$. Since 
	%$\mathbb{F}_{p_i}[x]$ is not a quotient of $R/I$, 
	$p_iR \nsupseteq I$,
	we find
	that $I$ contains polynomials $f_1, f_2, \ldots, f_t$ such that $\bar{f_i}$ in $\mathbb{F}_{p_i}[x]$ is not zero. Hence we may apply 
	Lemma \ref{lem:bound for primes not dividing content} again (for each $f_i$) to get bounds for the finitely many primes not included
	in the first paragraph. Taking the maximum of all the bounds finishes the proof.
\end{proof}

\begin{corollary}
	\label{cor:free abelian cyclic module has no growth - weaker version}
	Let $M = \integers^d$ be a cyclic $\integers[x]$-module. Then $\maxsubmod(M) \leq d$  
	for all $n$.
	%Let $M \tensor \rationals$ be a cyclic $\rationals[x,x^{-1}]$-module. Then $\maxsubmod(M) \leq d$ 
	%for all $n$.
\end{corollary}
\begin{proof}
	%By Lemmas \ref{lem:weaker version -- cyclic rationals module gives a cyclic integers module} and \ref{lem:weaker version -- finite index implies same exact growth},
	%we may assume (without changing the maximal submodule growth) that $M$ is a cyclic $\integers[x,x^{-1}]$-module.  
	We have that $M \cong_R R/I$ for some $I\ideal R$. Then $I$ contains the characteristic polynomial of $x$ (considered as a $\rationals$-linear transformation). 
	%Why not take the minimal polynomial?
	The result follows by Lemma \ref{lem:bound for primes not dividing content} since the characteristic polynomial is monic. 
	%Given any prime $p$, the polynomial $\bar{f}$ has at most $d$ 
	%factors in $\mathbb{F}_p[x]$. %$\mathbb{F}_p[x,x^{-1}]$. 
	%And so we are done (by Lemma \ref{lem:maximal ideals of integral polynomial ring}).
\end{proof}

We get the following: 

% To what extent can the following be extended? I know that we only need assume that $\rationals \tensor M$ is a cyclic $\rationals \tensor R$ module.
% Of course, we canNOT extend this to the case that G/M is free abelian of rank > 2. (And rank = 2 also probably won't work.)
\begin{corollary}
	If $N$ is a cyclic $\integers[x]$-module and $G = N \rtimes \integers = \freebycyclic{}$, then $\maxsubgr(G)$ has growth
	type $n$ and hence $\mdeg(G) = 1$. 
	%If $N \tensor \rationals$ is a cyclic $\rationals[x,x^{-1}]$-module and $G = N \rtimes \integers = \freebycyclic{A}$, then $\maxsubgr(G)$ has growth
	%type $n$ and hence $\mdeg(G) = 1$. 
\end{corollary}
\begin{proof}
	Just apply Lemma \ref{lem:abelian by infinite cyclic - exact counting} (which could have been proved in \S \ref{sec:counting derivations}) 
	together with Corollary \ref{cor:free abelian cyclic module has no growth - weaker version}
	to get that $\maxsubgr(G)$ has at most linear growth. 
	
	For the lower bound, notice that characteristic subgroups of the normal subgroup $N$ of $G$ must necessarily be
	normal in $G$. Note that the subgroups $pN$ (if the group operation in $N$ is written additively) or $N^p$ (if the group
	operation in $N$ were written multiplicatively), where $p$ is prime, are characteristic in $G$. Therefore $\maxsubmod(N) \geq 1$ for infinitely
	many $n$, and hence $\maxsubgr(G) \geq n$ for infinitely many $n$ (again by Lemma~\ref{lem:abelian by infinite cyclic - exact counting}).
	%Todo: Finish the lower bound argument as a separate lemma. 
	%Conjecture: Any f.g. infinite $\integers[x]$-module has infinitely many maximal submodules; equivalently, it has maximal submodules
	% of infinitely many different indices.
\end{proof}

\subsection{Finitely generated modules over PIDs}
\label{sec:Finitely generated modules over PIDs}

The PIDs considered in this section are all of the form $\mathbb{F}[x]$, where $\mathbb{F}$ is either $\finitefield$ or 
$\rationals$.

We first outline the main idea of this section. Let $N$ be a f.g.\ $ \integers[x]$ module.  Any maximal submodule of $N$
of index power of a prime $p$ will contain $pN$ and so corresponds to a maximal $\finitefield[x]$-submodule of $N/pN$. But since 
$\finitefield[x]$ is a PID, we can apply the structure theorem for f.g.\ modules over PIDs. If we only cared about the
prime $p$ (and no other primes), then we could immediately jump to \S\ref{sec:Direct sums with each term a quotient of the next}.
However, we do not care about only one specific prime. Rather, we want to know what happens for all (large) primes.

It would be computationally
advantageous if we did not need to apply the structure theorem infinitely many times---once for each prime. Indeed, 
one major goal of \S \ref{sec:global to local: from Q to F_p} is to prove
Lemma \ref{lem:version1_of_global to local - from QN to N/pN}, which
says that for all but finitely many primes, the decomposition of $N/pN$ afforded by the 
structure theorem (applied to the PID $\finitefield[x]$) ``comes from'' the decomposition
of $\rationals \tensor N$ as a $\rationals[x]$-module. The other major goal is to prove Lemma~\ref{lem:global to local - from QN to N/pN - improved},
a slight generalization of Lemma~\ref{lem:version1_of_global to local - from QN to N/pN}.

%We fix notation for what we get when we apply the structure theorem.  For each prime $p$, the module $N/pN$ can be written as
%a direct sum of cyclic $\finitefield[x]$-modules: % in such a way that each successive term is a quotient of the following term: 
%There exist integers $s(p), r(p)$ and
%polynomials (depending on $p$) $a_{1,p}, a_{2,p}, \ldots, a_{s(p),p} \in \finitefield[x]$, such that
%\[
% N/pN = \left( \bigoplus_{j=1}^{s(p)} \finitefield[x]/(a_{j,p}) \right) \oplus  (\finitefield[x])^{r(p)} %\oplus \finitefield[x]/(a_2) \oplus \cdots \oplus \finitefield[x]/(a_r) \oplus (\finitefield[x])^s
%\]
%where $a_{1,p} | a_{2,p} | \ldots | a_{s(p),p}$.
%
%When we instead apply the structure theorem for the $\rationals[x]$-module $\rationals \tensor N$, we replace the subscript $p$ by 0.

\subsection{Global to local: From $\rationals$ to $\finitefield$}
\label{sec:global to local: from Q to F_p}	
The goal of this section is to prove Lemma \ref{lem:version1_of_global to local - from QN to N/pN} (and its slight generalization).
It is possible that everything in this section is already known; certainly some of it is.

Until Lemma~\ref{lem:global to local - from QN to N/pN - improved}, let $N$ be a f.g.\ $\integers[x]$-module.
Denote $\rationals \tensor_\integers N$ by \nq. %, and suppose it is $n$-generated. 
Since $\rationals[x]$ is a PID, we have by the fundamental theorem of f.g.\ modules over PIDs that 
\[\tag{*}
\nq \cong_{\rationals[x]} \left( \bigoplus_{j=1}^{s_1} \rationals[x]/(a_j) \right) \oplus \rationals[x]^{s_2}
\]
for some $a_j \in \rationals[x]$ that are not units and such that $a_1 \mid a_2 \mid \ldots \mid a_{s_1}$.

Let $a \in \rationals[x]$. Then it is easy to see that for all large primes, we may speak of
$a$ mod $p$ and state that $\bar{a} \in \finitefield[x]$. Indeed,
there exists a finite set of primes $D$ such that 
$a \in \rd$.  Of course, \rd\ is a subring of $\rationals[x]$, and for $p \not\in D$ we have the surjection
$\rd \twoheadrightarrow \finitefield[x]$. 

What we need is to prove the following.
\begin{lemma}
	\label{lem:version1_of_global to local - from QN to N/pN}
	Suppose $N$ and \nq\ are as above. Then for all large primes $p$,
	\[
	\scrnp \cong_{\finitefield[x]} \left( \bigoplus_{j=1}^{s_1} \finitefield[x]/(\overline{a_j}) \right) \oplus \finitefield[x]^{s_2}.
	\]
\end{lemma}

We first give a high-level sketch of the basic idea. Then we state a slight generalization which we will need later. 
We then show how to give a proof by using
Lemma \ref{lem:D large enough gives a basis for ker(pi_D)} via 
Corollary \ref{cor:decomp_of_localization} (whose proofs are deferred to 
the end of this section). 
\begin{proof}[Sketch of proof idea.]
	When doing the computation required in finding the decomposition
	of $\rationals \tensor N$, the only thing keeping us from doing this computation to $N$ itself (as a $\integers[x]$-module) is that 
	we may need to divide by finitely many integers. 
	
	So if $p$ is large enough, then in $\finitefield$ we can divide by all those integers (i.e.\ their residues mod $p$). For such $p$, 
	the steps of the algorithm would be the same for $N/pN$ as for $\rationals \tensor N$. The way we fill out the details
	is to first pass from $\rationals[x]$ to a localization $\integers_D[x]$ of $\integers[x]$, where $D$ is finite. We then
	mod out by $p$.
\end{proof}

\begin{lemma}
	\label{lem:global to local - from QN to N/pN - improved}
	Suppose $N$ is a f.g.\ $\integers_{D_0}[x]$-module, where $D_0$ is a finite set of primes. 
	Also, suppose that 
	\[
	\nq \cong_{\rationals[x]} \left( \bigoplus_{j=1}^{s_1} \rationals[x]/(a_j) \right) \oplus \rationals[x]^{s_2}
	\]
	for some $a_j \in \rationals[x]$ that are not units and such that $a_1 \mid a_2 \mid \ldots \mid a_{s_1}$.
	Then for all large primes $p$,
	\[
	\scrnp \cong_{\finitefield[x]} \left( \bigoplus_{j=1}^{s_1} \finitefield[x]/(\overline{a_j}) \right) \oplus \finitefield[x]^{s_2}.
	\]
\end{lemma}
We next show how to prove Lemma~\ref{lem:global to local - from QN to N/pN - improved}, quoting a couple results which will be proved later.

Denote  $d_{\rationals[x]}(\nq)$ by $n$. So there
exists a surjection $\pi_\rationals\colon \rationals[x]^n \twoheadrightarrow \nq$ (which is a $\rationals[x]$-homomorphism). 
So $\ker (\pi_\rationals)$ is a 
$\rationals[x]$-submodule of the free module $\rationals[x]^n$. Because $\rationals[x]$ is a PID we conclude
that $\ker(\pi_\rationals)$ is a free module, and in fact we know that $\rationals[x]^n$ has a basis $y_1, y_2, \ldots, y_n$
such that $\ker(\pi_\rationals)$ has basis $b_1y_1, b_2y_2, \ldots, b_my_m$ for some $m \leq n$ such that 
$b_1 \mid b_2 \mid \ldots \mid b_{m}$. 

Claim 1: No $b_j$ is a unit. The reason is that because $d_{\rationals[x]}(\nq) = n$, there is no surjective $\rationals[x]$-module 
homomorphism from
$\rationals[x]^{n-1}$ to $\nq$.

Claim 2: We therefore have $m = s_1$, $n-m = s_2$, and for all $j = 1,\ldots,m$, $b_j = u_ja_j$ for $u_j$ a unit. The reason is that we have
\[
\nq  \cong_{\rationals[x]} \rationals[x]^n/(b_1y_1,\ldots, b_my_m), \text{ \quad and}
\]
\[
\rationals[x]^n/(b_1y_1,\ldots, b_my_m) \cong_{\rationals[x]} \left( \bigoplus_{j=1}^{m} \rationals[x]/(b_j) \right) \oplus \rationals[x]^{n-m}.
\]
Claim 2 then follows by the uniqueness of the decomposition afforded by the structure theorem. So from now on, we will write
$a_j$ instead of $b_j$.

We can make $D \supseteq D_0$ large enough (and yet keep it finite) such that $\pi_\rationals(y_i) \in \nd$.
In this case, there is a map 
$\pi_D$ %: \rd^n \twoheadrightarrow D^{-1}N$.
satisfying the following commutative diagram:% (where $\scrnd$ denotes $D^{-1}N$):

\centerline{\xymatrix{
		(\rd)^n \ar@{^{(}->}[d]^{\iota_1} \ar@{->>}[r]^{\pi_D} &\scrnd\ar@{^{(}->}[d]^{\iota_2}\\
		\rationals[x]^n \ar@{->>}[r]^{\pi_\rationals} 				&\nq}}

Note that if we have such $\pi_D$ and diagram for given $D$, then the same diagram holds (for a similar $\pi_D$)
if we make $D$ any larger.

Our main step in proving Lemma \ref{lem:global to local - from QN to N/pN - improved}
is Lemma \ref{lem:D large enough gives a basis for ker(pi_D)} 
which gives our main reduction, Corollary \ref{cor:decomp_of_localization}.
\begin{lemma}
	\label{lem:D large enough gives a basis for ker(pi_D)}
	With the above notation, we can make $D$ large enough yet finite such that
	$a_1y_1, a_2y_2, \ldots, a_my_m \in (\rd)^n$ and form a \rd-basis of $\ker(\pi_D)$. 
\end{lemma}

Once we prove this lemma, we will then get the following corollary, 
which tells us that our decomposition
for $\nq$ given at the beginning of the section passes to a decomposition
of the $\rd$-module $\scrnd$.

\begin{corollary}
	\label{cor:decomp_of_localization}
	For the above $D$, we have 
	\[
	\scrnd \cong_{\rd} \left( \bigoplus_{j=1}^{s_1} \rd/(a_j) \right) \oplus (\rd)^{s_2}.
	\]
\end{corollary}

Once we have this corollary, it will be straightforward to complete the proof of 
Lemma~\ref{lem:global to local - from QN to N/pN - improved}. Indeed, let $p \not\in D$. Then
\[
N/pN \cong_{\rd} D^{-1}(N/pN) \cong_{\rd} \scrnd/p\scrnd.
\]
Let $A$ denote the right-hand side of the isomorphism in Corollary \ref{cor:decomp_of_localization}. We have
\[
\scrnd/p\scrnd \cong_{\rd} A/pA \cong_{\rd} \left( \bigoplus_{j=1}^{s_1} \finitefield[x]/(\overline{a_j}) \right) \oplus \finitefield[x]^{s_2}.
\]
Combining the above two sequences of isomorphisms yields 
$$N/pN \cong_{\rd}  \left( \bigoplus_{j=1}^{s_1} \finitefield[x]/(\overline{a_j}) \right) \oplus \finitefield[x]^{s_2}$$
which passes to an isomorphism as $\finitefield[x]$-modules, giving
Lemma~\ref{lem:global to local - from QN to N/pN - improved} (and \ref{lem:version1_of_global to local - from QN to N/pN}).
The only thing that remains is to prove Lemma~\ref{lem:D large enough gives a basis for ker(pi_D)}
(and Corollary~\ref{cor:decomp_of_localization}).

\vspace{.1in}
\noindent \textbf{Proof of Lemma \ref{lem:D large enough gives a basis for ker(pi_D)}:}

\vspace{.05in}

\noindent To give a proof, we have to do some preliminaries first. 
Recall that a norm $\mathscr{N}$ on an integral domain $S$ is 
a function $\mathscr{N}: S \to \integers^{\geq 0}$ with $\mathscr{N}(0)  = 0$.
\begin{definition}
	Let $S$ be an integral domain with norm $\mathscr{N}$. Let $0 \neq b \in S$. We say that we can always divide by 
	$b$ in $S$ if for all $a \in S$, there exist $q, r \in S$ such that
	\[
	a = qb + r \text{\quad with $r = 0$ or $\mathscr{N}(r)  < \mathscr{N}(b)$}.
	\]
\end{definition}

\begin{lemma}
	\label{lem:leading coefficient a unit implies we can divide}
	Let $R$ be an integral domain and let $b(x) \in R[x]$. Then we can always divide by $b(x)$ in $R[x]$ if
	$\leadingcoeff(b)^{-1} \in R$.
\end{lemma}
\begin{proof}[Sketch of proof]
	This is clear by looking at the division algorithm in $\mathbb{F}[x]$, where $\mathbb{F}$ is the field
	of fractions of $R$.
\end{proof}

We know that in a Euclidean domain, every ideal is principal. In the process of showing that, we can extract a 
little more, namely Lemma~\ref{lem:minimal norm and ability to divide implies principal ideal}. 
%Recall that a norm $\mathscr{N}$
%on an integral domain $R$ is a function $\mathscr{N} \colon R \to \integers^{\geq 0}$ such that $\mathscr{N}(0) = 0$.
%\begin{definition}
%	Let $R$ be an integral domain with norm $\mathscr{N}$ and with $d \in R$. 
%	We say that ``we can always divide by $d$ in $R$'' if for every $a \in R$ 
%	there exist $q$, $r \in R$ such that 
%	\[
%	a = qd + r \text{\quad with } r = 0 \text{ or } \mathscr{N}(r) < \mathscr{N}(d).
%	\]
%\end{definition}
%
%With this terminology, we can say that a Euclidean domain is an integral domain 

\begin{lemma}
	\label{lem:minimal norm and ability to divide implies principal ideal}
	Let $R$ be an integral domain with norm $\mathscr{N}$. Suppose $I \ideal R$ and that there exists $d \in I$ such that
	\begin{enumerate}
		\item $\mathscr{N}(d) = \min_{0 \neq \alpha \in I}\{\mathscr{N}(\alpha) \}$ \text{ and}
		
		\item We can always divide by $d$ in $R$.
	\end{enumerate}
	Then $I = (d)$.
\end{lemma}
\begin{proof}
	We know $I \supseteq (d)$ since $d \in I$. 
	
	Showing $I \subseteq (d)$: Suppose that $a \in I$. Because we can divide by $d$, we know there exist
	$q, r \in R$ such that $a = qd + r$ with $r = 0$ or $\mathscr{N}(r) < \mathscr{N}(d)$. But $a, d \in I$ implies that $r \in I$ also.
	Therefore, by minimality of $\mathscr{N}(d)$, we conclude that $r = 0$. So $a \in (d)$.
\end{proof}

\begin{lemma}
	\label{lem:D exists for base case}
	Let $y_1, y_2, \ldots, y_n$ be a $\rationals[x]$-basis of $\rationals[x]^n$. Then there exists a finite $D$ (containing $D_0$) such that
	$y_1, y_2, \ldots, y_n$ form a $\integers_D[x]$-basis of $\integers_D[x]^n$.
\end{lemma}
\begin{proof}
	Let $e_1, e_2, \ldots, e_n$ be a $\integers_D[x]$-basis of $\integers_D[x]^n$. Thus $e_1, e_2, \ldots, e_n$ is a
	$\rationals[x]$-basis of $\rationals[x]^n$. For $i \in [n]$, let $\pi_i\colon \rationals[x]^n \to \rationals[x]$ be the projection onto
	the $i$-th coordinate: $\pi_i(\sum_j r_j e_j) := r_i$. Fix $k \in [n]$. Then $y_k \in \integers_D[x]^n$ iff 
	for all $i \in [n]$, we have $\pi_i(y_k) \in  \integers_D[x]$.
	
	Therefore, there exists a finite $D$ (containing $D_0$) such that $y_k \in  \integers_D[x]^n$ for all $k$, but this is not sufficient.
	
	We have that $y_1, y_2, \ldots, y_n$ is a basis for $ \integers_D[x]^n$ iff the map $e_i \mapsto y_i$  $\forall i \in [n]$ is an isomorphism. 
	We note that this map ($e_i \mapsto y_i$) is given by a matrix; indeed, for given $j$, let
	$y_j = \sum_{i=1}^n a_{ij} e_i$, and form the $n \times n$ matrix $A := (a_{ij})$. Of course, the entries of $A$ are all
	in $\integers_D[x]$.  %(so that the $j$-th column of $A$ is $[a_{1j}, a_{2j}, \ldots, a_{nj}]^T$). 
	
	The matrix $A$ is invertible in the ring $M_n(\integers_D[x])$ (of all $n \times n$ matrices over $\integers_D[x]$)
	iff the map $A \cdot$ from $ \integers_D[x]^n$ to $ \integers_D[x]^n$ ``multiply on the left by $A$'' is an isomorphism. Also,
	the map $A \cdot$ is an isomorphism iff $y_1, y_2, \ldots, y_n$ is a $\integers_D[x]$-basis of $\integers_D[x]^n$.
	We have that $A$ is an invertible matrix iff $\det(A)$ is a unit in $\integers_D[x]$. Because $y_1, y_2, \ldots, y_n$
	is a $\rationals[x]$-basis of $\rationals[x]^n$, we have that $0 \neq \det(A) \in \rationals$. Therefore, we can make $D$
	large enough (while keeping it finite) so that $\det(A)^{-1} \in \integers_D[x]$.	
\end{proof}

\begin{lemma}
	\label{lem:lem:D large enough gives a basis for ker M_D}
	Let \mq\  denote some $\rationals[x]$-submodule of $\rationals[x]^n$, and denote $\qm \cap \zdx^n$ by \md.
	Let $y_i \in \rationals[x]^n$, $a_i \in \rationals[x]$ be such that $y_1, y_2, \ldots, y_n$ is a 
	$\rationals[x]$-basis of $\rationals[x]^n$ and $a_1y_1, \ldots, a_m y_m$ is a $\rationals[x]$-basis of
	\mq.
	% we know such a basis exists because the heart of the structure theorem for f.g.\ modules over PIDs is that 
	% a submodule of a f.g. free module over a PID is itself free.
	Then there exists a finite $D \supseteq D_0$ such that 
	$a_1y_1, \ldots, a_my_m$ is a $\zdx$-basis of \md.
\end{lemma}
\begin{proof}
	We will show that there exists a finite $D$ such that for all $c \in \md$, there exist unique $r_1, \ldots, r_m \in \zdx$ such that
	$c = r_1 a_1 y_1 + \cdots + r_m a_m y_m$. 
	
	Suppose by Lemma \ref{lem:D exists for base case} that $D \supseteq D_0$ is large enough (yet finite) such that
	$y_1, \ldots, y_n$ is a \zdx-basis of $\zdx^n$. (So $\zdx^n = \bigoplus_{i = 1}^n \zdx y_i$.) For $i \in [n]$, 
	let $\pi_i\colon \rationals[x]^n \to \rationals[x]$ be given by $\pi_i(\sum_j r_j y_j) := r_i$.
	
	We know $\pi_i(\mq) \ideal \rationals[x]$ is a principal ideal (since $\rationals[x]$ is a PID), 
	but more, we have $\pi_i(\mq) = (a_i)_{\rationals[x]} :=$ the ideal of $\rationals[x]$ generated by $a_i$.
	
	Add if necessary, finitely many primes to $D$ such that for all $i$, $a_i \in \zdx$ and such that
	$\leadingcoeff(a_i)^{-1} \in \integers_D$. Consequently, Lemma \ref{lem:leading coefficient a unit implies we can divide}
	tells us we can always divide by $a_i$ in \zdx. Since $\pi_i(\mq) = (a_i)_{\rationals[x]}$ we have that $a_i$ has minimal 
	degree in $\pi_i(\mq)$, and hence $a_i$ also has minimal degree in $\pi_i(\md)$. Therefore,
	we conclude by Lemma \ref{lem:minimal norm and ability to divide implies principal ideal} that 
	$\pi_i(\md) = (a_i)_{\zdx} :=$ the ideal of $\zdx$ generated by $a_i$.  We now have $D$ picked.
	
	Let $c \in \md$. We know that since $c \in \mq$ and since $a_1 y_1, \ldots, a_m y_m$ is a $\rationals[x]$-basis of \mq, 
	there exist unique $c_1, \ldots, c_m \in \rationals[x]$ such that 
	\[\tag{*}
	c =  c_1 a_1 y_1 + \cdots + c_m a_m y_m.
	\]
	
	We have that $\pi_i(c) \in \pi_i(\md) = (a_i)_{\zdx}$. Therefore, for all $i$, there exist $d_i \in \zdx$ such that
	$\pi_i(c) = d_i a_i$. But we know from (*) that $\pi_i(c) = c_i a_i$. Thus $c_i a_i = d_i a_i$. Since
	\zdx\ is an integral domain, we conclude that $c_i = d_i$ for all $i$. Hence $c_i \in \zdx$, and they are unique.
\end{proof}

\begin{proof}[Proof of Lemma \ref{lem:D large enough gives a basis for ker(pi_D)}]
	Notice that in the notation of the commutative diagram preceding 
	Lemma \ref{lem:D large enough gives a basis for ker(pi_D)} that
	$$\ker(\pi_D) = \ker(\iota_2 \circ \pi_D) = \ker(\pi_\rationals \circ \iota_1) = \ker(\pi_\rationals) \cap (\rd)^n.$$
	Let $M_\rationals = \ker(\pi_\rationals)$ and $M_D = \ker(\pi_\rationals) \cap (\rd)^n$, and
	apply Lemma \ref{lem:lem:D large enough gives a basis for ker M_D}.
\end{proof}

\begin{proof}[Proof of Corollary \ref{cor:decomp_of_localization}]
	%	(This really is standard. See the proof of the Thm.\ 5 in \cite{Dummit2004}, pg.\ 462-463.)
	%	Again, we have 
	%	$\ker(\pi_D) = \ker(\pi_\rationals) \cap (\rd)^n.$ Of course, $N_D \cong_{\zdx} (\zdx)^n/\ker(\pi_D)$.
	%	If any $a_i$ from Lemma \ref{lem:D large enough gives a basis for ker(pi_D)} is a unit, we may remove it 
	%	(and decrement $n$); there will be $m - s_1$ units (where $s_1$ is from the decomposition of $\nq$ and $m$ is from 
	%	Lemma \ref{lem:D large enough gives a basis for ker(pi_D)}). 
	We have $\pi_D\colon \rd^n \twoheadrightarrow N_D$ from the commutative diagram preceding Lemma~\ref{lem:D large enough gives a basis for ker(pi_D)}.
	Therefore,
	\[\tag{*1}
	N_D \cong_{\rd }  \rd^n/ \ker(\pi_D).
	\]
	We have by Lemma~\ref{lem:D large enough gives a basis for ker(pi_D)} that
	\[\tag{*2}
	\rd^n/ \ker(\pi_D) \cong_{\rd} \rd^n/(a_1y_1, \ldots, a_my_m).
	\]
	Since $y_1, \ldots, y_n$ form a basis of $\rd^n$ (as Lemma~\ref{lem:D exists for base case} says), we get that
	\[\tag{*3}
	\rd^n/(a_1y_1, \ldots, a_my_m) \cong_{\rd } \left( \bigoplus_{j=1}^{m} \rd/(a_j) \right) \oplus (\rd)^{n-m}
	\]
	By Claim 2 (which follows the statement of Lemma~\ref{lem:global to local - from QN to N/pN - improved}) we have $m = s_1$ and $n - m = s_2$.
	Thus, combining (*1), (*2), and (*3) gives Corollary~\ref{cor:decomp_of_localization}.
\end{proof}
\noindent This completes our proof of Lemma~\ref{lem:global to local - from QN to N/pN - improved} and hence of
Lemma~\ref{lem:version1_of_global to local - from QN to N/pN} too.

Let $N$ be a f.g.\ $\integers_{D_0}[x]$-module (for some finite set of primes $D_0$). 
Denote the $\finitefield[x]$-torsion-free rank of $N/pN$ by $r(p)$ and write 
\[
\nq \cong_{\rationals[x]} \left( \bigoplus_{j=1}^{s(0)} \rationals[x]/(a_j) \right) \oplus \rationals[x]^{r(0)}.
\]
\begin{corollary}
	\label{cor:p^k approaches infinity in two ways - consequence on module}
	With the notation from the previous paragraph, there exists a constant $C$ (depending on $N$) such that $n = p^k > C$ implies
	\begin{enumerate}[(a)]
		\item $\maxsubmod(N) = \maxsubmod(\finitefield[x]^{r(p)})$ \hspace{.1in} or 
		\item $\displaystyle{\maxsubmod(N) =  \maxsubmod \left( \bigoplus_{j=1}^{s(0)} \finitefield[x]/(\overline{a_{j}}) \oplus \finitefield[x]^{r(0)} \right) }$.
	\end{enumerate}
\end{corollary}
\begin{proof}
	%Let $n = p^k$. Thus $\maxsubmod(N) = \maxsubmod(N/pN)$. 
	By Lemma~\ref{lem:global to local - from QN to N/pN - improved}, for all large primes $p$ we have
	\[ 
	N/pN = \left(\bigoplus_{j=1}^{s(0)} \finitefield[x]/(\overline{a_{j}}) \right) \oplus \finitefield[x]^{r(0)}.
	\]
	Let the exceptions, if any, be $\{p_1, p_2, \ldots, p_s\}$. 
	Let $p$ be any such prime. We know that the $\finitefield[x]$-torsion part of $N/pN$ is finite; say its cardinality is $p^{c_p}$. 
	Thus for $k > c_p$, we have $\maxsubmod[p^k](N) = \maxsubmod[p^k](\finitefield[x]^{r(p)})$ 
	
	Let $c = \max \{c_p : p \in \{p_1, p_2, \ldots, p_s \} \}$, and let $C = p_s^{c}$ (assuming $p_s$ is the biggest among $p_1, p_2, \ldots, p_s$). Then this $C$ works. (If there were no exception primes $p_i$, then of course (b) always holds.)
\end{proof}

\subsection{Direct sums with each term a quotient of the next}
\label{sec:Direct sums with each term a quotient of the next}
This subsection is a continuation of Section \ref{sec:Codimension 1 subspaces}. However, 
it fits naturally here because a finitely generated module over a PID can be written in the form described in the next paragraph.

Let $R$ be a commutative (unital) ring. %either the ring $\integers[x]$ or $\finitefield[x]$. 
Let $A = A_1 \oplus A_2 \oplus \cdots \oplus A_t$, where each
$A_j$ is a cyclic $R$-module such that $A_j$ is a quotient of $A_{j+1}$ for $j = 1, 2, \ldots, t-1$.  Fix a positive integer $n$,
and let $\mathcal{S}\mathcal{Q}^n$ be \emph{any} set of \textbf{s}imple \textbf{q}uotients of $A_1$ of index $n$.

\begin{lemma}
	\label{lem:each term quotient of the next -- take some simple quotients of first term}
	Using the notation from the preceding paragraph, we have
	\[
	\sum_{S \in \mathcal{S}\mathcal{Q}^n} \maxsubmod[S](A) = |\mathcal{S}\mathcal{Q}^n|(1 + n + \cdots + n^{t-1}).
	\]
\end{lemma}
\begin{proof}
	Let $S \in \mathcal{S}\mathcal{Q}^n$. Because $A_1$ is a quotient of $A_j$ for all $j \in \{2, 3, \ldots, t\}$, we conclude that
	$|\{A_j : \maxsubmod[S](A_j) = 1 \}| = t$. Therefore, Corollary \ref{cor:submodisoto R/I for direct sum of cyclic modules} says that
	$\maxsubmod[S](A) = 1 + n + \cdots + n^{t-1}$. 
\end{proof}

\begin{corollary}
	\label{cor:each term quotient of the next -- take certain simple quotients of j-th term}
	Let $A$ be as in Lemma \ref{lem:each term quotient of the next -- take some simple quotients of first term}. Fix $j \in \{1, 2, \ldots, t\}$.
	Let $\mathcal{S}\mathcal{Q}_j^n$ be a set of simple quotients of $A_j$ of index $n$ such that
	$\maxsubmod[S](A_i) = 0$ for $i < j$. Then
	\[
	\sum_{S \in \mathcal{S}\mathcal{Q}_j^n} \maxsubmod[S](A) = |\mathcal{S}\mathcal{Q}_j^n|(1 + n + \cdots + n^{t-j}).
	\]
\end{corollary}
\begin{proof}
	Let $S \in \mathcal{S}\mathcal{Q}_j^n$. Then $\maxsubmod[S](A) = \maxsubmod[S](A_j \oplus A_{j+1} \oplus \cdots \oplus A_t$). The result then
	follows from Lemma \ref{lem:each term quotient of the next -- take some simple quotients of first term} by reindexing (by subtracting
	$(j-1)$ from each index in $A_j, A_{j+1}, \ldots, A_t$).
\end{proof}

We fix a little more notation for the following lemma. 
Let $A_0$ be the zero $R$-module. For an $R$-module $B$, let $\mathcal{S}\mathcal{Q}(B,n)$ be the set of \emph{all} simple
quotients of $B$ of index $n$. 

\begin{corollary}
	\label{cor:each term quotient of the next -- exact number of maximal submodules}
	Using the notation from the paragraph preceding this corollary and the paragraph before Lemma \ref{lem:each term quotient of the next -- take some simple quotients of first term},
	we have
	\[
	\maxsubmod(A) = \sum_{j=1}^{t} (\maxsubmod(A_j) - \maxsubmod(A_{j-1}))(1 + n + \cdots + n^{t-j}).
	\]
\end{corollary}
\begin{proof}
	The idea is just to write $\mathcal{S}\mathcal{Q}(A,n)$ as a disjoint union as follows (which we can do
	since it is assumed that $A_j$ is a quotient of $A_{j+1}$ for all $j$):
	\[ \tag{*}
	\mathcal{S}\mathcal{Q}(A,n) = \bigsqcup_{j=1}^{t} (\mathcal{S}\mathcal{Q}(A_j,n) \setminus \mathcal{S}\mathcal{Q}(A_{j-1},n)).
	\]
	Let $\mathcal{S}\mathcal{Q}_j^n := \mathcal{S}\mathcal{Q}(A_j,n) \setminus \mathcal{S}\mathcal{Q}(A_{j-1},n)$. We have (with explanations following)
	\[\begin{aligned}
	\maxsubmod(A) &= \sum_{S \in \mathcal{S}\mathcal{Q}(A,n)} \maxsubmod[S](A) \\
	&= \sum_{j=1}^{t} \sum_{S \in \mathcal{S}\mathcal{Q}_j^n} \maxsubmod[S](A) \\
	&= \sum_{j=1}^{t} (\maxsubmod(A_j) - \maxsubmod(A_{j-1}))(1 + n + \cdots + n^{t-j}).
	\end{aligned}
	\]
	The first equality is by Lemma \ref{lem:maxsubmod N = sum_S submodisoto_S(N)}. The second equality is by equation (*).
	For the third equality, recall that in a \emph{cyclic} module $B$, two maximal submodules $M_1$ and $M_2$ of $B$ are  equal iff  
	$B/M_1 \cong_R B/M_2$. In other words, for a cyclic module $B$, we have $|\mathcal{SQ}(B,n)| = \maxsubmod(B)$.
	Thus $|\mathcal{S}\mathcal{Q}_j^n| = \maxsubmod(A_j) - \maxsubmod(A_{j-1})$ because each $A_i$ is cyclic (and since
	$\mathcal{SQ}(A_{j-1},n) \subseteq \mathcal{SQ}(A_{j},n)$).
	Thus the third equality
	follows by Corollary \ref{cor:each term quotient of the next -- take certain simple quotients of j-th term}. 
\end{proof}

\begin{corollary}
	\label{cor:each term quotient of the next -- upper and lower bounds}
	Using the notation from the paragraph proceeding Lemma \ref{lem:each term quotient of the next -- take some simple quotients of first term} we
	have for all $n$,
	\[\begin{aligned}
	\maxsubmod(A) & \leq \maxsubmod(A_t)(1 + n + \cdots + n^{t-1}) \text{\hspace{.1in} and }\\
	\maxsubmod(A) & \geq \maxsubmod(A_1)(1 + n + \cdots + n^{t-1}).
	\end{aligned}
	\]
\end{corollary}
\begin{proof}
	For the second inequality, the lower bound for $\maxsubmod(A)$, just note that the first term in the sum in 
	Corollary \ref{cor:each term quotient of the next -- exact number of maximal submodules} is 
	$\maxsubmod(A_1)(1 + n + \cdots + n^{t-1})$; of course, all the other terms in the sum of that corollary are non-negative.
	
	For the first inequality, the upper bound for $\maxsubmod(A)$, we use Corollary 
	\ref{cor:each term quotient of the next -- exact number of maximal submodules} again to get
	\[\begin{aligned}
	\maxsubmod(A) & =    \sum_{j=1}^{t} (\maxsubmod(A_j) - \maxsubmod(A_{j-1}))(1 + n + \cdots + n^{t-j}) \\
	& \leq \sum_{j=1}^{t} (\maxsubmod(A_j) - \maxsubmod(A_{j-1}))(1 + n + \cdots + n^{t-1}) \\
	& = \maxsubmod(A_t)(1 + n + \cdots + n^{t-1}).   
	\end{aligned}
	\]
\end{proof}

Note that this corollary does not give us the maximal submodule growth of such a module $A$ because $A$ itself may be finite; in case $A$ is finite,
we would have
$\maxsubmod(A) = \maxsubmod(A_1) = 0$ for all large $n$.

\subsection{$\integers_D[x]$-modules which are f.g.\ as $\integers_D$-modules}
\label{sec:Z_D[x]-modules which are f.g. as Z_D-modules}

Let $R$ = $\integers_D[x]$ (for some finite $D$), and let $N$ be an $R$-module which is finitely generated as a $\integers_D$-module. 
Suppose 
\[
\rationals \tensor N =  \bigoplus_{j=1}^d \rationals[x]/(a_j),
\]
where $a_1 | a_2 | \cdots | a_d$ as provided by the structure theorem with $a_1$ (and hence each $a_i$) not a unit. 
We have then that $d = d_{\rationals \tensor R}(\rationals \tensor N)$ is the minimal size of a $\rationals \tensor R$ generating set for
the module $\rationals \tensor N$. The following lemma 
extends Lemma \ref{lem:version1_of_global to local - from QN to N/pN}:

We now state how Corollary \ref{cor:p^k approaches infinity in two ways - consequence on module} simplifies
since $N$ is assumed to be f.g.\ as a $\integers_D$-module.
\begin{corollary} %deleted
	\label{cor:we can ignore finitely many primes - for f.g. abelian groups}
	Using the above notation, there exists a constant $C$ (depending on $N$) such that $n = p^k > C$ implies either
	\begin{enumerate}[(a)]
		\item $\maxsubmod(N) = 0$ \hspace{.1in} or 
		\item $\displaystyle{N/pN =  \bigoplus_{j=1}^d \finitefield[x]/(\overline{a_j})}$.
	\end{enumerate}
\end{corollary}
\begin{proof}
	This follows from Corollary \ref{cor:p^k approaches infinity in two ways - consequence on module}. Because
	$N$ is f.g.\ as a $\integers_D$-module, then for all primes $p$, $\finitefield[x]$ is \emph{not} a quotient of 
	$N/pN$. %I think this follows because $\rationals N$ is a finite dimensional vector space over $\rationals$.
	%We know by Lemma \ref{lem:version1_of_global to local - from QN to N/pN} that (b) holds for all large primes.
	%Let the exceptions, if any, be $\{p_1, p_2, \ldots, p_s\}$.
	%
	%Let $p$ be any prime, and suppose $N$ is $d$-generated as an abelian group. 
	%Then $|N/pN| \leq p^d$, and hence $C = p_s^d$ works.	
	
	% 
	% Let $p$ be any prime. We know that $N/pN$ is finite (since $N$ is a f.g.\ abelian group). Therefore there exists a constant $c_p$ (depending on $p$)
	% such that $\maxsubmod[p^k](N) = \maxsubmod[p^k](N/pN) = 0$ for all $k \geq c_p$. %Maybe justify the first equality somewhere... 
	% % I think currently (as of June 14th, 2016), I never stated it as a lemma. The closest thing now would be Lemma \ref{lem:R/I of N to R/I of N/IN}.
	% 
	% Let $c = \max \{c_p : p \in \{p_1, p_2, \ldots, p_s \} \}$, and let $C = p_s^{c}$ (assuming $p_s$ is the biggest among $p_1, p_2, \ldots, p_s$). Then this $C$ works.
\end{proof}

Although the following could be taken as a corollary to Theorem \ref{thm:exact growth -with coefficient- of Z x module which is f.g. abelian},
we include a proof of this simpler result because it is easier.
\begin{proposition}
	\label{prop:mdeg(N) -actually tilde{m}deg(N)}
	With the above notation, $\maxsubmod(N)$ has growth type $n^{d-1}$, where still
	$d = d_{\rationals \tensor R}(\rationals \tensor N)$.
\end{proposition}
\begin{proof}
	Let $C$ be as in
	Corollary \ref{cor:we can ignore finitely many primes - for f.g. abelian groups}. Fix $n = p^k > C$,
	such that $\maxsubmod(N) \neq 0$. Then by Corollary  \ref{cor:we can ignore finitely many primes - for f.g. abelian groups}, we conclude 
	that $\displaystyle{ N/pN =  \bigoplus_{j=1}^d \finitefield[x]/(\overline{a_j})}$. 
	
	We show that $\maxsubmod(N) \leq \deg(a_d)(1 + n + \cdots + n^{d-1})$ for all large $n$: By 
	Corollary \ref{cor:each term quotient of the next -- upper and lower bounds}, we get 
	\[  \tag{*}
	\maxsubmod(N/pN) \leq \maxsubmod(\finitefield[x]/(a_d))(1 + n + \cdots + n^{d-1}).
	\]
	But by Lemma \ref{lem:upper bound for finite F_p x module},
	we get that $\maxsubmod(\finitefield[x]/(\overline{a_d})) \leq \deg(\overline{a_d}) \leq \deg(a_d)$. Notice that the constant $\deg(a_d)$ 
	does not depend on which (large) $p$ we pick. This gives us the 
	desired upper bound.
	
	For the lower bound, by Corollary \ref{cor:each term quotient of the next -- upper and lower bounds} we get:
	\[
	\maxsubmod(N/pN) \geq \maxsubmod(\finitefield[x]/(\overline{a_1}))(1 + n + \cdots + n^{d-1}).
	\]
	Conclude by noting that since $a_1$ is a non-constant polynomial, 
	we get that $\maxsubmod[p^k](\finitefield[x]/(\overline{a_1})) \geq 1$ for some $k$.% where recall that $n = p^k$.
\end{proof}

We can be a bit more precise than simply stating the growth type of $\maxsubmod(N)$. 
% To do so, 
% consider the polynomials $a_j$ from the paragraph before Corollary~\ref{cor:we can ignore finitely many primes - for f.g. abelian groups}. Let
% $c_j$ be the number of distinct roots of $a_j$ in $\overline{\finitefield}$. Fix a large prime $p$ so that
% $\displaystyle{ N/pN =  \bigoplus_{j=1}^d \finitefield[x]/(\overline{a_j})}$. Notice that 

\begin{theorem}
	\label{thm:exact growth -with coefficient- of Z x module which is f.g. abelian}
	%Let $\epsilon > 0$. 
	Let $N$, $d$, and $a_j$ be as in the beginning of Section \ref{sec:Z_D[x]-modules which are f.g. as Z_D-modules}, and
	let $\rho_j$ be the 
	number of distinct roots of $a_j$ in $\C$. Then
	\[\begin{aligned}
	\maxsubmod(N) & \leq \rho_1 n^{d-1} + O(n^{d-2}) & \text{for all large $n$, and}\\
	\maxsubmod(N) & \geq \rho_1n^{d-1}                     & \text{for infinitely many $n$.} 
	\end{aligned} 
	\]
\end{theorem}
\begin{proof}
	Upper bound: 
	
	Fix a large $n = p^k$,
	such that $\maxsubmod(N) \neq 0$; by Corollary~ \ref{cor:we can ignore finitely many primes - for f.g. abelian groups}, we conclude 
	that $$ N/pN =  \bigoplus_{j=1}^d \finitefield[x]/(\overline{a_j}).$$ % Let $c_j = \deg(a_j)$.
	We first show
	\[ \tag{*}
	\maxsubmod(N) \leq \sum_{j=1}^d \rho_j(1 + n + \cdots + n^{d-j}).
	\]
	Indeed, we just use 
	Corollary \ref{cor:each term quotient of the next -- exact number of maximal submodules} together with 
	Lemma \ref{lem:upper bound for finite F_p x module} part (b): Let 
	$A_i = \finitefield[x]/(\overline{a_{i}})$. We have,
	\[
	\maxsubmod(A_j) - \maxsubmod(A_{j-1}) \leq \maxsubmod(A_j), %\leq \rho_j,
	\]
	and by Lemma~\ref{lem:upper bound for finite F_p x module}, $\maxsubmod(A_j)$ is bounded above by
	the number of roots of $a_j$ in $\overline{\finitefield}$, 
	which by Lemma~\ref{lem:num_roots in overline(F_p) of polynomial bounded above by num_roots in C}, is bounded above by $\rho_j$. This
	shows (*). %Here is where the commented part below was earlier, when the upper bound was (\rho_1 + \epsilon)n^{d-1}
	Let `$RHS$' denote the right hand side of (*). Then
	\[
	\begin{aligned}
	RHS &= \sum_{k=0}^{d-1} \left(\sum_{j=1}^{d-k} \rho_j \right) n^k \\
	&= \rho_1 n^{d-1} + \sum_{k=0}^{d-2} \left(\sum_{j=1}^{d-k} \rho_j \right) n^k.
	\end{aligned}
	\]
	Combining these equalities with (*) completes the upper bound.
	
	% Notice that
	% \[
	%  \sum_{j=1}^d \rho_j(1 + n + \cdots + n^{d-j}) \leq \rho_1n^{d-1} + \sum_{j=1}^d \rho_j(1 + n + \cdots + n^{d-2}).
	% \]
	% Let $\rho = \sum_{j=1}^d \rho_j$. Simplifying the right hand and using (*) yields
	% \[ %\tag{*2
	%  \maxsubmod(N) \leq \rho_1n^{d-1} + \rho(1 + n + \cdots + n^{d-2}).
	% \]
	% This concludes the upper bound because 
	% \[
	%  \rho(1 + n + \cdots + n^{d-2}) = \frac{\rho}{n-1}(n^{d-1}-1).
	% \]
	Lower bound: 
	
	%It can be shown by Chebatorev's theorem that there are infinitely many primes $p$ such that $a_1$ splits, in which case
	By Lemma~\ref{lem:roots of non-constant polynomials in F_p and in C},
	there are infinitely many primes $p$ such that 
	$\maxsubmod[p](A_1) = \rho_1$, where $A_1 = \finitefield[x]/(\overline{a_{1}})$ as above. We conclude by using
	Corollary \ref{cor:each term quotient of the next -- upper and lower bounds}.
\end{proof}

\subsection{General f.g.\ $\integers[x]$-modules}
Again, let $R = \integers[x]$. In this subsection, we do \emph{not} assume that our modules are f.g.\ as abelian groups. 
We begin by stating a result that could have been given in 
Section \ref{sec:Direct sums with each term a quotient of the next}:

\begin{corollary}
	\label{cor:maxsubmod A^d = maxsubmod A times 1 + n + cdots + n^ d-1}
	Let $A$ be a cyclic $R$-module, and let $d \in \integers_{\geq 1}$. Then
	\[
	\maxsubmod(A^d) = \maxsubmod(A)(1+ n + \cdots + n^{d-1}).
	\]
\end{corollary}
\begin{proof}
	This follows immediately from Corollary \ref{cor:each term quotient of the next -- exact number of maximal submodules}
	(or Corollary \ref{cor:each term quotient of the next -- upper and lower bounds}).
\end{proof}

\begin{corollary}
	\label{cor:maxsubmod 'Z x' ^ d  has growth type n^d}
	Fix $d \in \integers_{\geq 1}$. Still $R = \integers[x]$. Then $\maxsubmod(R^d)$ has growth type $n^d$. 
	In fact, $\maxsubmod(R^d) \leq n^d +n^{d-1} + \cdots + n$ for all $n$, with equality
	when $n$ is prime. 
\end{corollary}
\begin{proof}
	This follows from Lemma \ref{lem:maximal ideal growth of Z x} and Corollary \ref{cor:maxsubmod A^d = maxsubmod A times 1 + n + cdots + n^ d-1}.
\end{proof}

We give another consequence of Corollary \ref{cor:maxsubmod A^d = maxsubmod A times 1 + n + cdots + n^ d-1}:
\begin{corollary}
	\label{cor:maxsubmod 'finitefield x' ^d has growth type n^d / log n}
	Fix $d \in \integers_{\geq 1}$. Then
	\[
	\maxsubmod[p^k](\finitefield[x]^d) = \frac{1}{k}\sum_{a|k} \mu\left(\frac{k}{a}\right)p^a (1 + p^k + p^{2k} + \cdots + p^{(d-1)k}).
	\]
	Thus $\maxsubmod[p^k](\finitefield[x]^d)$ has growth type $n^d/\log(n)$.
\end{corollary}
\begin{proof}
	This follows from Lemma \ref{lem:mobious function - exact number of irreducibles} 
	and Corollary \ref{cor:maxsubmod A^d = maxsubmod A times 1 + n + cdots + n^ d-1}. (For the ``growth type,''
	the reader may also want to recall Corollary \ref{cor:mmoddeg of finitefield[x]}.)
\end{proof}

Let $N$ be an $R$-module. For any prime $p$, we will use the notation from 
Corollary \ref{cor:p^k approaches infinity in two ways - consequence on module}
to denote the $\finitefield[x]$-torsion-free rank of $N/pN$ by $r(p)$, and recall
\[
\nq \cong_{\rationals[x]} \left( \bigoplus_{j=1}^{s(0)} \rationals[x]/(a_j) \right) \oplus \rationals[x]^{r(0)}.
\]
%(Also, recall Definition \ref{def:growth type} for our usage of $\Omega$ and ``growth type''.)

%However, we will write $\Tor(N/pN)$ for the $\finitefield[x]$-torsion part: $\bigoplus_{j=1}^{s(p)} \finitefield[x]/(a_{j,p}$). 
%Recall that the $\finitefield[x]$-torsion part of $N/pN$ is $\bigoplus_{j=1}^{s(p)} \finitefield[x]/(a_{j,p})$ and that
%$r(p)$ is the rank of the $\finitefield[x]$-torsion-free part of $N/pN$. 
%Notice that by Lemma \ref{lem:mod out by finite submodule -- maximal submodule growth is unchanged},
%for large $k$, $\maxsubmod[p^k](N/pN) = \maxsubmod[p^k](\finitefield[x]^{r(p)})$, and we could then
%say what this equals using Corollary \ref{cor:maxsubmod 'finitefield x' ^d has growth type n^d / log n}.

\begin{proposition}
	\label{prop:growth type of arbitry f.g. Z[x]-module}
	With the notation from the previous paragraph, let $d = s(0) + r(0)$, and $r_{\text{max}} = \max_{p} \{r(p) \}$. 
	\[
	\maxsubmod(N) \text{ has growth type }
	\begin{cases}
	n^{d-1} & \text{if } d > r_{\text{max}} \\
	n^d & \text{if } d = r_{\text{max}} = r(0) \\
	n^{r_{\text{max}}}/\log(n) & \text{otherwise}
	\end{cases}
	\]
	%If $d > r_{\text{max}}$, then $\maxsubmod(N)$ has growth type $n^{d-1}$. Otherwise, $\maxsubmod(N)$ has 
	%growth type $n^{r_{\text{max}}}/\log(n)$.
\end{proposition}
\begin{proof}
	The basic idea is that for large $n = p^k$, either $k$ or $p$ is large (or both). We then apply 
	Corollary \ref{cor:p^k approaches infinity in two ways - consequence on module}. The growth type of
	$\maxsubmod(N)$ will be controlled by one of two things:
	\begin{enumerate}[(i)]
		\item Fix $p$ such that $r(p) = r_{\text{max}}$. Let $k$ approach infinity (in $n = p^k$). Or\ldots
		
		\item  Keep $k$ small and send $p$ to infinity.
	\end{enumerate} 
	% If $k$ is large (and $p$ bounded), then
	% $\maxsubmod[p^k](N/pN) = \maxsubmod[p^k](\finitefield[x]^{r(p)})$. If $p$ is large, then the decomposition of $N/pN$
	% comes from that of $\rationals \tensor N$.
	% 
	We define three auxiliary functions that will simplify our proof:
	\[\begin{aligned}
	f_1(n) & = n^{d-1} \\
	f_0(n) & = n^{r(0)} \\
	g_0(n) &= n^{r_{\text{max}}}/\log(n)
	\end{aligned}\]
	Also, we will decompose the function $\maxsubmod(N)$ into two parts. Let $C$ be the constant given by 
	Corollary \ref{cor:p^k approaches infinity in two ways - consequence on module}. Define $f$ and $g$ as follows.
	First, $f(n) = 0$ and $g(n) = 0$ if $n$ is not a power of a prime. For a prime power $p^k$, 
	\[
	f(p^k) :=
	\begin{cases}
	\maxsubmod[p^k](N) & \text{if } p > C \\
	0 								  & \text{otherwise}
	\end{cases}
	\]
	\[
	g(p^k) :=
	\begin{cases}
	\maxsubmod[p^k](N) & \text{if } p \leq C \\
	0 								  & \text{otherwise}
	\end{cases}
	\]
	We have then that $\maxsubmod(N) = f(n) + g(n)$ for all $n$. 
	%Hence, $\maxsubmod(N)$ has growth type $h$, where
	%$h(n) = \max\{f(n), g(n)\}$. %This is irrelevant, since $h(n) = \maxsubmod(N)$ !! 
	% Notation not used: $\prec$ and $\preceq$
	
	\emph{Claim 1}: $g$ has growth type $g_0$.
	Indeed, there are only finitely many primes $p$ for which for some $k$,
	$g(p^k) \neq 0$. We then apply Corollary \ref{cor:p^k approaches infinity in two ways - consequence on module} (a)
	and then Corollary~\ref{cor:maxsubmod 'finitefield x' ^d has growth type n^d / log n} for each prime. This proves Claim 1.
	
	\emph{Claim 2}: $f$ has growth type $f_0$ if $s(0) = 0$ and growth type $f_1$ if $s(0) \geq 1$. 
	
	Case $s(0) = 0$: By our choice of $C$, for all primes $p > C$ we have
	$N/pN = \finitefield[x]^{r(0)}$. Hence $\maxsubmod[p](N/pN) = \sum_{i=1}^{r(0)} p^{i}$. So 
	$f$ has growth type at least $f_0$. Also, since  $\maxsubmod(\finitefield[x]^{r(0)}) \leq \maxsubmod(\integers[x]^{r(0)})$,
	Corollary \ref{cor:maxsubmod 'Z x' ^ d  has growth type n^d} implies that $f$ has growth type at most
	$f_0$. Hence $f$ has growth type $f_0$, finishing the case $s(0) = 0$.
	
	Case $s(0) \geq 1$: %(which happens if $d > r_{\text{max}} $)
	Assume that $p > C$. 
	Then by Lemma \ref{lem:version1_of_global to local - from QN to N/pN}, we get
	\[
	N/pN = \left( \bigoplus_{j=1}^{s(0)} \finitefield[x]/(\overline{a_{j,0}}) \right) \oplus  (\finitefield[x])^{r(0)}.
	\]
	
	Each term in the $s(0) + r(0)$ terms of the direct sum decomposition of $N/pN$ is a quotient of the 
	next one (except for the last 
	$\finitefield[x]$, since there is no term after it). Letting $A_{1,p} = \finitefield[x]/(\overline{a_{1,0}})$,
	by Corollary \ref{cor:each term quotient of the next -- upper and lower bounds} we get
	\[
	\maxsubmod(N/pN)  \geq \maxsubmod(A_{1,p})(1 + n + \cdots + n^{d-1}).
	\]
	Because $a_{1,0}$ is not constant, we get that for some $k$, $\maxsubmod[p^k](A_{1,p}) \geq 1$. Therefore, $f$ has 
	growth type at least $f_1$.
	
	Also, we have (with explanations following)
	\[\begin{aligned}
	\maxsubmod(N/pN) & \leq \maxsubmod(A_{1,p} \oplus \finitefield[x]^{d-1}) \\
	& \leq \maxsubmod(A_{1,p})(1 + n + \cdots + n^{d-1}) + \maxsubmod(\finitefield[x])(1+ n + \cdots + n^{d-2} ) \\
	& \leq (\maxsubmod(A_{1,p})+1)(1 + n + \cdots + n^{d-1})\\
	& \leq (\deg(a_{1,0})+1)(1 + n + \cdots + n^{d-1}).
	\end{aligned}
	\]
	The first inequality is because $N/pN$ is a quotient of $A_{1,p} \oplus \finitefield[x]^{d-1}$. The second  is
	by Corollary \ref{cor:each term quotient of the next -- exact number of maximal submodules}. The third 
	is because $\maxsubmod(\finitefield[x]) \leq n$. The fourth
	is because $\maxsubmod(A_{1,p}) \leq \deg(\overline{a_{1,0}}) \leq \deg(a_{1,0})$. Notice that combining
	these inequalities gives a bound for $f(n)$ independent of which large prime $p$ we use. Therefore, $f$ has 
	growth type at most $f_1$ and therefore has growth type $f_1$. This finishes the case $s(0) \geq 1$ and proves
	Claim 2.
	
	Note that for all large $n$, one of $f_0, f_1, g_0$ will be asymptotically at least as big as the other two. Hence we just
	need to decide which is biggest given the different cases in this proposition.
	
	Suppose that $d > r_{\text{max}}$. Then $r_{\text{max}} \leq d - 1$. Hence, $g_0(n) \leq f_1(n)$ for $n \geq 2$. Further, we always
	have $r(0) \leq r_{\text{max}}$. Combining this with the previous inequality gives $r(0) \leq d-1$. Therefore,
	$f_0(n) \leq f_1(n)$ for all $n$. We just showed that $f_1$ is asymptotically largest among $\{f_0, f_1, g_0\}$. Note that 
	because $d > r_{\text{max}}$ implies that $s(0) \geq 1$, $f$
	has growth type $f_1$ by Claim 2 above. We conclude that
	$\maxsubmod(N)$ has growth type $f_1(n) = n^{d-1}$.
	
	Next, suppose that $d = r_{\text{max}} = r(0)$. Then $f_1(n) < f_0(n)$ and $g_0(n) \leq f_0(n)$ 
	for $n \geq 2$. So $f_0(n) = n^d$ is asymptotically largest among $\{f_0, f_1, g_0\}$. We observe that
	$d = r_{\text{max}}$ implies that $s(0) = 0$. Hence, Claim 2 above shows that $f$ has growth type $f_0$.
	Therefore, $\maxsubmod(N)$ has growth type $f_0(n) = n^d$.
	
	Finally, suppose that $d \leq r_{\text{max}}$ and that either $d \neq r_{\text{max}}$ or $d \neq r(0)$. Then
	$f_1(n) \leq g_0(n)$ for $n \geq 2$. We show next that $r(0) < r_{\text{max}}$. Indeed,
	notice that if $d \neq r_{\text{max}}$, then $d < r_{\text{max}}$ and
	hence $r(0) < r_{\text{max}}$. Also, if $d \neq r(0)$, then $s(0) \geq 1$, in which case
	$d = r(0) + s(0) \leq r_{\text{max}}$ implies $r(0) < r_{\text{max}}$. So whether, $d \neq r_{\text{max}}$
	or $d \neq r(0)$, we get $r(0) < r_{\text{max}}$. Hence, $f_0(n) < g_0(n)$ for $n \geq 2$. Combining this with the
	second sentence of this paragraph, we see that $g_0$ is largest among $\{f_0, f_1, g_0\}$. So $f$ has growth type
	at most $g_0$. Also, Claim 1 says that $g$ has growth type $g_0$. Therefore $\maxsubmod(N)$ has
	growth type $g_0(n) = n^{r_{\text{max}}}/\log(n)$. 
\end{proof}

\section{$\integers[x_1, x_2, \ldots, x_\ell]$-modules, f.g.\ as abelian groups}
\label{sec:fin gen modules over polynomial ring with several variables}

%In this section, we mainly deal with modules that are also finitely generated as abelian groups.

Fix a positive integer $\ell$. Let $R = \integers[x_1, x_2, \ldots, x_\ell]$. Note that $\rationals \tensor_\integers R$ (which we will often
denote as $\rationals R$)
is just $\rationals[x_1, x_2, \ldots, x_\ell]$. The difficulty dealing with f.g.\ $R$ modules is that we have more than one variable.
Consequently, $\rationals[x_1, x_2, \ldots, x_\ell]$ is not a
principal ideal domain. Thus, we do not have the nice structure theorem which was so useful to us when we had only one variable. We should
not lose heart, however, since if we restrict to $ \integers[x_1, x_2, \ldots, x_\ell]$-modules that are f.g.\ as abelian groups, then
we can basically summarize the action of all $\ell$ variables with a \emph{single} variable. We can then apply the structure theorem as before.

\subsection{Reducing to one variable}
% Remember that Mazur gave the essential idea of Lemma \ref{lem:a surjection from rationals[x] to A/Jac(A)}.

Let $N = \integers^k$ be a $\intvars$-module.  Consider $\rationals N := \rationals \tensor_\integers N = \rationals^k$.
Then for each $i$, the map $\rationals^k \to \rationals^k$ given by $x_i \cdot$ (i.e.\ multiplication by $x_i$)
is a linear transformation. Let $f_i$ be the minimal polynomial of $x_i \cdot$, and let $A = \qvars/(f_1,\ldots,f_\ell)$. 
Let $\jrad$ be the Jacobson radical of $A$. Fix polynomials $g_1,\ldots, g_s \in \qvars$ such that 
$\jrad = (g_1,\ldots,g_s)_{A}$. 
We may assume that in fact $g_i \in \intvars$ for all $i$, because if they were not in $\intvars$, then we could scale each by an integer
and rename them. Lemma \ref{lem:a surjection from rationals[x] to A/Jac(A)} was shown to the author by Marcin Mazur.

\begin{lemma}
	\label{lem:a surjection from rationals[x] to A/Jac(A)}
	There exists a surjection $\pi_\rationals\colon \rationals[x] \twoheadrightarrow A/\jrad$.
\end{lemma}
\begin{proof}
	Because $A$ is a finite dimensional algebra over $\rationals$, we have that $A/\jrad$ is semisimple.\footnote{See for example, 
		Lemma~6.3.1 part (2) in \cite{Webb}, which states that if a module $U$ satisfies the descending chain condition on submodules, then
		$U/\Rad(U)$ is semisimple. We are dealing with semisimple \emph{algebras}, but that is fine, by Proposition~0.10 in 
		\url{http://www.ucl.ac.uk/~ucahaya/SemisimpleModules.pdf}}
	Since $A$ is commutative, $A/\jrad$ is a product of fields (each a finite extension of $\rationals$):
	\[\tag{*1}
	A/\jrad \cong \prod_{j = 1}^n F_i.
	\]
	But not only is every number field a simple extension of $\rationals$ (by the primitive element theorem), 
	but for each finite extension $E$ of $\rationals$ we have that $\{\alpha : \rationals(\alpha)  = E \}$ is infinite. 
	% (I think Mazur said that the set of generators is a Zariski open set.) 
	So choose $\alpha_j$ with $F_j = \rationals(\alpha_j)$ such that different $\alpha_j$'s have different minimal polynomials.
	Let $m_j(x)$ be the minimal polynomial of $\alpha_j$. Thus we can restate (*1) as 
	\[\tag{*2}
	A/\jrad \cong \prod_{j = 1}^n \rationals[x]/(m_j(x)).
	\]
	We may then apply the Chinese remainder theorem and conclude that 
	\[\tag{*3}
	\prod_{j = 1}^n \rationals[x]/(m_i(x)) \cong \rationals[x]/(m_1(x)\cdots m_n(x)).
	\]
	Of course, there is a surjection from $\rationals[x]$ to $\rationals[x]/(m_1(x)\cdots m_n(x))$. (This surjection, combined with
	(*2) and (*3), finishes the proof.)
\end{proof}

For any finite set $D$ of primes (to be decided later), let
\[
B := \intDvars/(f_1,\ldots,f_\ell).
\]
Recall that $\jrad = (g_1,\ldots,g_s)_{\qvars}$ with each $g_i \in \intvars$. Let 
\[
\jrad' := (g_1,\ldots,g_s)_{B}.
\]

%Here is how to avoid a page break (or page-break or pagebreak - I'm repeating myself in case I hit Cmd-F and search for the word)
\vbox{
	\begin{lemma}
		\label{lem:D exists for a commutative diagram}
		With the above notation, there exists a finite $D$ with the following commutative diagram
	\end{lemma}
	\centerline{\xymatrix{
			\rd \ar@{^{(}->}[d]^{\iota_1} \ar@{->>}[r]^{\pi_D} &B/\jrad'\ar@{^{(}->}[d]^{\iota_2}\\ %I was thinking of calling the second map \rationals \tensor instead of \iota_2
			\qr \ar@{->>}[r]^{\pi_\rationals} 				&A/\jrad}}
}
\begin{proof}
	Let $f \in \qvars$ be such that $\pi_\rationals(x) = \bar{f}$. 
	We can make $D$ large enough (and yet keep it finite) such that $f \in \intDvars$. Of course, we then define
	$\pi_D \colon \zdx \to B/\jrad'$ via $\pi_D(x) := \bar{f} \in B/\jrad'$. 
	
	What we need to do next is to make $D$ large enough to ensure that $\pi_D$ is surjective. 
	Because $\pi_\rationals$ is surjective, we know that $\bar{f}$ is a generator for $A/\jrad$. 
	So for each $\bar{x}_i \in A/\jrad$, 
	choose $a_{i,j} \in \rationals$ such that 
	\[
	\bar{x}_i = \sum_{j=0}^{n_j} a_{i,j} (\bar{f})^j.
	\]
	Choose $D$ such that $a_{i,j} \in \integers_D$ for all $i$,$j$. This ensures that $\pi_D$ is surjective.
\end{proof}

Let $N$ be the module defined at the beginning of this section, and let $D$ 
be as in Lemma~\ref{lem:D exists for a commutative diagram}. Let $A$ be from the paragraph before
Lemma~\ref{lem:a surjection from rationals[x] to A/Jac(A)}. Let $B$ and $\jrad'$ be from the paragraph 
before Lemma~\ref{lem:D exists for a commutative diagram}. (So $B = \intDvars/(f_1,\ldots,f_\ell)$.)
 We are given that $N$ is a $\integers[x_1,\ldots,x_\ell]$-module. 
So $N$ is a $\integers[x_1,\ldots,x_\ell]/(f_1,\ldots,f_\ell)$-module.
Thus $\integers_DN = \integers_D \tensor_\integers N$ is a $B$-module.

\begin{lemma}
	\label{lem:maximal submodules contain jrad' Z_D N}
  We use the notation from the previous paragraph. Also, let $S = \integers_D N$. 
  Let $M$ be a maximal submodule of $S$.
  Then $M$ contains $\jrad' S$.
\end{lemma}
\begin{proof}
	$A$ is a $\rationals$-algebra of finite dimension over $\rationals$. Therefore $A$ is Artinian. Consequently, $\jrad$ is a nilpotent ideal of $A$.
	Therefore, $\jrad'$ is a nilpotent ideal of $B$. % See for example the following:
	%https://qchu.wordpress.com/2012/05/30/the-jacobson-radical/
	%This is the blog titled The Jacobson radical, posted on May 30, 2012 by Qiaochu Yuan.
	%I found this blog by googling the following phrase:
	%the jacobson radical of an artinian ring is nilpotent
	
	By contradiction, suppose that $M$ does not contain $\jrad' S$. Then $M + \jrad' S = S$. By induction, suppose that for some $k \geq 1$
	that $M + (\jrad')^k S = S$. Then $S = M + \jrad' S = M + \jrad' (M + (\jrad')^k S) = M + \jrad' M + (\jrad')^{k+1} S = M + (\jrad')^{k+1} S.$ And so
	we have shown that $S = M + (\jrad')^{k+1} S$. Therefore, for all $n \geq 1$, we have that $M + (\jrad')^n S = S$. But since $\jrad'$ is a nilpotent
	ideal, by taking $n$ to be large enough, we have that $M = S$, a contradiction.
\end{proof}
A proof of Lemma~\ref{lem:maximal submodules contain jrad' Z_D N} was shown to the author by Marcin Mazur.

\begin{lemma} %Todo: ask Mazur if this is good.
	\label{lem:Z_DN/JZ_DN - many variables to one}
	Let $N$ be the module defined at the beginning of this section, and let $D$ 
	be as in Lemma~\ref{lem:D exists for a commutative diagram}. Let $S = \integers_D N.$ Then $S /\jrad'S$ 
	is a $\integers_D[x]$-module such that if $p \notin D$ and $k \geq 1$, then
	\[
	\maxsubmod[p^k](S /\jrad'S) = \maxsubmod[p^k](N).
	\]
\end{lemma}
\begin{proof}%\footnote{The format of this proof is to make it easier to read.}
	%We are given that $N$ is a $\integers[x_1,\ldots,x_\ell]$-module. 
	%\newline So $N$ is a $\integers[x_1,\ldots,x_\ell]/(f_1,\ldots,f_\ell)$-module.
	%\newline Thus $S = \integers_D \tensor_\integers N$ is a $B$-module,
	%\newline $\text{\hspace{1in}}$ where $B = \integers_D[x_1,\ldots,x_\ell]/(f_1,\ldots,f_\ell)$.
	As stated in the paragraph before Lemma~\ref{lem:maximal submodules contain jrad' Z_D N}, $S$ is a $B$-module.
	Hence, $S /\jrad'S$ is a $B$-module too, and so $S /\jrad'S$ is a  $B/\jrad'$-module. 
	Therefore, $S /\jrad'S$ is a $\integers_D[x]$-module by 
	Lemma~\ref{lem:D exists for a commutative diagram}.
	
	We have that $\integers_D[x]$-submodules of $S /\jrad'S$ are the same as
	$B/\jrad'$-submodules of $S /\jrad'S$, and these are the same as $B$-submodules of $S /\jrad'S$.
	Also, $B$-submodules of $S /\jrad'S$ are in one-to-one correspondence with
	$B$-submodules of $S$ that contain $\jrad'S$, and by Lemma~\ref{lem:maximal submodules contain jrad' Z_D N}, 
	this includes all maximal submodules.
	Next, $B$-submodules of $S$ of finite index are in one-to-one correspondence to 
	\newline $\integers[x_1,\ldots,x_\ell]/(f_1,\ldots,f_\ell)$-submodules of $N$ of index relatively prime to everything in $D$.
	Finally, $\integers[x_1,\ldots,x_\ell]/(f_1,\ldots,f_\ell)$-submodules of $N$ are in one-to-one correspondence to
	$\integers[x_1,\ldots,x_\ell]$-submodules of $N$.
	Therefore, if $p \notin D$ and $k \geq 1$, then
	\[
	\maxsubmod[p^k](S /\jrad'S) = \maxsubmod[p^k](N).
	\]
\end{proof}

\begin{lemma}%Lemma (*1) and (*2) from my notes 5/2/17
	\label{lem:several variables to one, with a nice direct sum too}
	Let $N$ be a $\integers[x_1,\ldots,x_\ell]$-module which is f.g.\ as an abelian group. There exists a finite $D$ and a 
	module, denoted $\tilde{N}_{D}$, such that
	\[
	\tilde{N}_{D} = \bigoplus_{i = 1}^{d_0} \integers_{D}[x]/(a_i),
	\]
	for some $a_1 \mid a_2 \mid \cdots \mid a_{d_0}$ ($a_1$ not a unit) and such that for all large $n$,
	\[\tag{*}
	\maxsubmod(N) = \maxsubmod(\tilde{N}_{D}).
	\]
\end{lemma}
\begin{proof}
	By Lemma~\ref{lem:mod out by finite submodule -- maximal submodule growth is unchanged}, we may mod out by the finite submodule of 
	$N$ consisting of its $\integers$-torsion. So assume $N = \integers^k$. 
	By Lemma~\ref{lem:Z_DN/JZ_DN - many variables to one}, there exists $N_{D_0}$ with
	\[
	N_{D_0} = \integers_{D_0} N /\jrad'\integers_{D_0}N
	\]
	such that $\maxsubmod(N_{D_0}) = \maxsubmod(N)$ for all $n$. 
	By Corollary~\ref{cor:decomp_of_localization}, there exists $D \supseteq D_0$ such that if we localize $N_{D_0}$ by $D$, then
	we can write the resulting module as a direct sum.
	
	The final statement in the present lemma follows from
	Lemma~\ref{lem:Z_DN/JZ_DN - many variables to one} together with a fact about localization. Since we localize by a larger $D$,
	in order for the equation (*) to hold, we want the index $n = p^j$ be %\footnote{As usual, since $N_{D_0}$ is f.g.\
	%as a $\integers_{D_0}[x]$-module, if $j$ is large, then both $\maxsubmod(N)$ and $\maxsubmod(\tilde{N}_{D_0})$ are 0.}
	such that $p \notin D$. (If $j$ is large, $\maxsubmod(N) = 0 = \maxsubmod(\tilde{N}_{D})$.)
\end{proof}

\begin{lemma}
	\label{lem:size of QR generating set for QN is d_0}
  Let $N$, $\tilde{N}_{D}$, and $d_0$ be as in Lemma \ref{lem:several variables to one, with a nice direct sum too}. Let 
  $R = \integers[x_1,\ldots,x_\ell]$. Then
  \[
   d_{\rationals R}(\nq) = d_{\rationals R}(\rationals \tilde{N}_D) = d_0.
  \]
\end{lemma}
\begin{proof}
	For ease of notation let $N' = \tilde{N}_D$. We have that $\rationals N' = \bigoplus_{i = 1}^{d_0} \rationals[x]/(a_i)$
	for some $a_1 \mid a_2 \mid \cdots \mid a_{d_0}$ ($a_1$ not a unit).
	Therefore $d_{\rationals R}(\rationals N) = d_0$.
	
	Reviewing the proof of Lemma \ref{lem:several variables to one, with a nice direct sum too} we have that
	$N' = \integers_D N/\jrad'\integers_D N$ for some finite set of primes $D$. Recall $A$ and $\jrad$ from the paragraph
	before Lemma~\ref{lem:a surjection from rationals[x] to A/Jac(A)}. We have that 
	$\rationals N' = \rationals N/\jrad' \rationals N = \rationals N/\jrad \rationals N.$ ($\rationals N$ is an $A$-module. So 
	$\rationals N'$ is too.) We have $d_{\rationals R}(\rationals N) = d_A(\rationals N)$. Also, 
	$d_{\rationals R}(\rationals N') = d_{\rationals R}(\rationals N/\jrad \rationals N ) = d_A(\rationals N/\jrad \rationals N )$.
	So all we need is to show that $d_A(\rationals N) = d_A(\rationals N/ \jrad \rationals N)$, but this follows from
	Nakayama's Lemma.
\end{proof}

\subsection{Isolating the `trivial' part}
In order to state a simple formula for the maximal subgroup growth of groups of the form $\integers^k \rtimes \integers^\ell$ 
(and more general semidirect products), 
we introduce some notation.\footnote{As it turns out, we will also use this notation for modules arising from virtually abelian groups.} 

\begin{definition}
	\label{def:mtriv and mnontriv}
	Let $G$ be a group (or commutative monoid) and $N$ a f.g.\ $G$ module.
	\begin{enumerate}[(a)]
		\item $\mtriv(N)$ denotes the number of index $n$ maximal submodules $M$ of $N$ such that the action of 
		$G$ on $N/M$ is trivial.\footnote{By this, we mean that $g \cdot (n + M) = n + M$ for all $n \in N$ and all $g \in G$.}
		\item $\mnontriv(N)$ denotes the number of index $n$ maximal submodules $M$ of $N$ such that the action of 
		$G$ on $N/M$ is non-trivial.\footnote{By this, we mean that there exists $g \in G$ and $n \in N$ such that  $g \cdot (n+M) \neq n + M$.}
	\end{enumerate}
\end{definition}

Note that of course, $\mtriv(N) + \mnontriv(N) = \maxsubmod(N)$. 

Before continuing, it may be good to point out what $G$ usually is. If $N$ arises as a normal subgroup of a metabelian group, such as the
$\integers^k$ in
$\integers^k \rtimes \integers^\ell$, then $G = \integers^\ell$. Similarly, if $N$ is a $\integers[x]$ module, then the monoid $G$ is 
$\langle x \rangle = \{ x, x^2, x^3, \ldots \}$.
Also, though not used in this paper, if $N$ is an abelian normal subgroup of finite index in a virtually abelian group, then
the $G$ in Definition~\ref{def:mtriv and mnontriv} would be the finite quotient. With this in mind, the reader is encouraged to at least read the statements of
Lemmas~\ref{lem:abelian by free abelian - num of maximal subgroups} and \ref{lem:abelian by f.g. abelian - num of maximal subgroups}
before continuing this section.
% I formerly had the file separating-trivial-and-nontrivial here. It's now in the folder "Old stuff or not included", which itself is in the writing folder.

\begin{lemma}
	\label{lem:mtriv(N) = maxsubmod(N/IN) = maxsubgr(N/IN)}
	Let $N$ be a $\integers[x_1, x_2, \ldots, x_\ell]$-module which is f.g.\ as a $\integers$-module. Let
	$I = (x_1 - 1, x_2 - 1, \ldots, x_\ell - 1)$. Then
	\[
	\mtriv(N) = \maxsubmod(N/IN) = \maxsubgr(N/IN).
	\]
\end{lemma}
\begin{proof}
	Let $M$ be a maximal submodule of $N$ such that $x_i(n + M) = n + M$ for all $i$ and all $n \in N$. This means that
	$(x_i - 1) n \in M$ for all $i$ and $n$. In other words $IN \subseteq M$. Thus $M$ is counted in the term $\mtriv(N)$ iff
	$IN \subseteq M$. Therefore, $\mtriv(N) = \maxsubmod(N/IN)$.
	
	The equality $\maxsubmod(N/IN) = \maxsubgr(N/IN)$ follows because the trivial action by all the $x_i$'s implies that
	maximal submodules of $N/IN$ are the same thing as maximal subgroups of $N/IN$.
\end{proof}

\begin{corollary}
	\label{cor:mtriv(N) is (n^t-1)/(n-1) for n prime}
	Let $N$ and $I$ be as in Lemma~\ref{lem:mtriv(N) = maxsubmod(N/IN) = maxsubgr(N/IN)}. Let $t$ be the torsion-free rank of
	(the abelian group) $N/IN$. Then for all large $n$,
	\[
	\mtriv(N) = \maxsubmod(N/IN) =
	\begin{cases}
	\frac{n^t - 1}{n - 1} &\text{ if $n$ is prime}\\
	0 							 &\text{ otherwise.}
	\end{cases}
	\]
\end{corollary}
\begin{proof}
	This follows from Lemma~\ref{lem:mtriv(N) = maxsubmod(N/IN) = maxsubgr(N/IN)} together with two more facts.
	First, by Lemma~\ref{lem:mod out by finite submodule -- maximal submodule growth is unchanged}, we may mod out
	by the $\integers$-torsion part of $N/IN$ to get
	\[
	\maxsubgr(N/IN) = \maxsubgr(\integers^t) \text{\quad for all large } n.
	\]
	Second, 
	\[
	\maxsubgr(\integers^t) =
	\begin{cases}
	\frac{n^t - 1}{n - 1} &\text{ if $n$ is prime}\\
	0 							 &\text{ otherwise.}
	\end{cases}
	\]
\end{proof}

%DELETED:
%\subsection{Another proof of (most of) the upper bound}
%This section does not quite fit here. As the upper bound is only part of the upper bound found in
%the main theorem: Theorem \ref{thm:(f.g. abelian) by (f.g. abelian)}.

\section{Certain metabelian groups}
\label{sec:certain metabelian groups}

\subsection{Semidirect products}

Except for part of Theorem~\ref{thm:(f.g. abelian) by (f.g. abelian)} and two lemmas, the groups that appear in this section are
semidirect products.

% Is the following in the right place now?
In the following lemma, $N$ is a module over the group-ring $\integers[\integers] = $ Laurent polynomials $\integers[x,x^{-1}]$, 
where multiplication by $x$ (in the module) is conjugation (in $G$) by a chosen generator $x$ of $\integers$.
Recall that $\maxsubmod(N)$ denotes the number of maximal submodules of a module $N$ of index $n$. 
\begin{lemma}
	\label{lem:abelian by infinite cyclic - exact counting}
	Let $G = N \rtimes \integers$ be a f.g.\ group with $N$ abelian. Then 
	\[
	\maxsubgr(G) = \maxsubgr(\integers) + n \cdot \maxsubmod(N). 
	\]
\end{lemma}
\begin{proof}
	This follows immediately upon combining Lemma \ref{lem:abelian by something - counting maximal subgroups by derivations}
	and Lemma \ref{lem:Shalev's counting derivations out of cyclic groups}.
\end{proof}

%Todo: Why is it that if G if a finitely generated metabelian group such that N is an abelian normal subgroup with G/N abelian,
%that N is a finitely generated Z[G/N] module?
Notes: (1) For any group $G$, if there is a group $N \normal G$ such that $G/N \cong \integers$, 
then the extension splits, as 
in the hypothesis of the Lemma \ref{lem:abelian by infinite cyclic - exact counting}. (2) The function 
$\maxsubgr(\integers)$ is the characteristic function of the prime numbers and hence is always either 1 or 0. As a result,
$\maxsubgr(\integers)$ does not effect the growth rate of $\maxsubgr(N \rtimes \integers)$.

\begin{theorem}
	\label{thm:maximal subgroup growth of (f.g. abelian) by Z}
	Let $G = N \rtimes \integers$, with $N$ f.g.\ as an abelian group. Let 
	\[
	\rationals \tensor_\integers N =  \bigoplus_{j=1}^d \rationals[x]/(a_j),
	\]
	where $a_1 | a_2 | \cdots | a_d$ as provided by the structure theorem (so with $a_1$ not a unit). (So $d = d_{\rationals[x]}(\rationals N)$.)
	Also, let $\rho_1$ be the number of 
	(distinct) roots of $a_1$ in $\C$. Then
	\[\begin{aligned}
	\maxsubgr(G) & \leq \rho_1 n^{d} + O(n^{d-1}) & \text{for all large $n$, and}\\
	\maxsubgr(G) & \geq \rho_1n^{d} & \text{for infinitely many $n$.} 
	\end{aligned} 
	\]
\end{theorem}
\begin{proof}
	This follows from Theorem~\ref{thm:exact growth -with coefficient- of Z x module which is f.g. abelian} together with
	Lemma~\ref{lem:abelian by infinite cyclic - exact counting}.
\end{proof}
%Technically speaking, this theorem assumes that $d \geq 1$. In other words, $N$ is not finite.

\begin{corollary}
	\label{cor:mdeg(N rtimes A) = d(QN)}
	Suppose that $N$ is f.g.\ as an abelian group. Then 
	\[
	\mdeg(N \rtimes \integers) = d_{\rationals[x]}(\rationals N).
	\]
\end{corollary}
\begin{proof}
	This follows from Theorem~\ref{thm:maximal subgroup growth of (f.g. abelian) by Z}.
\end{proof}

\begin{corollary}
	\label{cor:max subgr growth of Z^k semidir_A Z}
  Let  $G = \integers^k \rtimes_A \integers$, where $A \in GL(k,\integers)$. Let
  $b$ = the number of 
  	blocks in the rational canonical form of $A$, and let
  $\rho_1$ = the number of distinct roots (in $\C$) of the characteristic polynomial of the smallest block. Then
  \[\begin{aligned}
  \maxsubgr(G) & \leq \rho_1 n^{d} + O(n^{d-1}) & \text{for all large $n$, and}\\
  \maxsubgr(G) & \geq \rho_1 n^{d} & \text{for infinitely many $n$.} 
  \end{aligned} 
  \]
\end{corollary}
\begin{proof}
	This follows from Theorem~\ref{thm:maximal subgroup growth of (f.g. abelian) by Z}.
\end{proof}

We next give a few examples of Corollary~\ref{cor:mdeg(N rtimes A) = d(QN)} in the form of another
corollary.
%{cor:two full lattices in Q^d - one contained in the other => the index is finite}

\begin{corollary}
	\label{cor:Z rtimes_sigma Z - for arbitrary permutation - mdeg = number of cylces}
	Let $\sigma \in \Sym(k)$, and suppose $\sigma$ has $c$ cycles. Then
	\[
	\mdeg(\integers^k \rtimes_\sigma \integers) = c.
	\]
\end{corollary}
\begin{proof}
	By Corollary~\ref{cor:mdeg(N rtimes A) = d(QN)}, all we need to show is that $d_{\rationals[x]}(\rationals N) = c$. In the proof, as usual, 
	we will denote the abelian normal subgroup $\integers^k$ by $N$.
	
	We first show that $d_{\rationals[x]}(\rationals N) \geq c$. Indeed, let the $c$ cycles of $\sigma$ be of lengths $n_1$,\ldots,$n_c$. Then
	\[\tag{*}
	\rationals N \cong \bigoplus_{i=1}^{c} \rationals[x]/(x^{n_i}-1).
	\]
	Because $x-1$ divides $x^{n_i} - 1$, we notice that there exists a submodule of $\rationals^k$  of the form $\rationals^c$, where
	the action of $x$ on $\rationals^c$  is trivial. Hence $d_{\rationals[x]}(\rationals N) \geq \dim_\rationals(\rationals^c) = c$.
	
	To see that $d_{\rationals[x]}(\rationals N) \leq c$, all we need to note is that (*) says $\rationals N$ is a direct sum of 
	$c$ cyclic modules. 
\end{proof}

\begin{corollary}
	For any $k \geq 1$ there exists a f.g.\ group $G$ and a finite index subgroup $H$ such that $\mdeg(G) = 1$ while $\mdeg(H) = k$.
\end{corollary}
\begin{proof}
	Let $G = \integers^k \rtimes_\sigma \integers$, where $\sigma \in \Sym(k)$ is a $k$-cycle. Let 
	$H = \integers^k \rtimes k\integers$, which equals $\integers^k \times \integers$. Then $\mdeg(\integers^{k+1}) = k$, but
	by Corollary~\ref{cor:Z rtimes_sigma Z - for arbitrary permutation - mdeg = number of cylces}, $\mdeg(G) = 1$.
\end{proof}
Note that this is in stark contrast to what happens when working with \emph{all} subgroups of a group. Theorem 1.1 from
Shalev's \cite{Shalev_On_the_degree} says that if $G$ is a f.g.\ group with $H \leq_f G$, then $\deg(G) \leq \deg(H) +1$.

We would like to give a perhaps more group theoretic interpretation of 
Corollary~\ref{cor:mdeg(N rtimes A) = d(QN)}. With this in mind, we make an observation
on Corollary \ref{cor:Z rtimes_sigma Z - for arbitrary permutation - mdeg = number of cylces}.
When considering the group $\integers^k \rtimes_\sigma \integers$, it is easy to find a set of $c$ elements which normally generate
$\integers^k$ (equivalently, which generate $\integers^k$ as a $\integers[x,x^{-1}]$-module). Indeed, $\sigma$ partitions $[n]$ into $c$ cycles. Let $i_1,\ldots,i_c$ be 
a complete set of representatives of the cycles. We already have a basis $e_1,\ldots,e_k$ of $\rationals^k = \rationals \tensor_\integers \integers^k$ fixed
(which in fact generates $\integers^k$ as a $\integers$-module). 
We conclude
that the elements $e_{i_1},\ldots,e_{i_c}$ normally generate $\integers^k$.

\begin{corollary}
	Let $G = N \rtimes \integers$, with $N$ f.g.\ as an abelian group.
	Let $n$ be the minimal number of elements of $N$ whose \textbf{n}ormal closure in $G$ has finite index in $N$. Then
	\[
	\mdeg(G) = n.
	\]
\end{corollary}
\begin{proof} %TODO: does this proof need to have Z[x,x^{-1}] (and Q[x,x^{-1}]) almost everywhere instead of Z[x] etc.?
	Let $d = d_{\rationals[x]}(\rationals N)$. %Every set of normal generators is a set of Q[x] generators. Hence $n \geq d$.
	Let $B = \{a_1, \ldots, a_k \} \subseteq N$. The normal closure of $B$ (in $G$), denoted 
	$\overline{\langle B \rangle}_G$, is the \emph{same set} as the $\integers[x]$-submodule of $N$ that $B$ generates. %Recall that the action of multiplication by
	% x in the module is just conjugation by x in the group.
	Let $N_0 = \overline{\langle B \rangle}_G$. If $N_0 \leq_f N$, then $\rationals N_0 = \rationals N$. Therefore, 
	$n \geq d$.
	
	% To every set of Q[x]-generators, we produce a set of normal generators. Hence $n \leq d$.
	To show that $n \leq d$, we will just point out how every set of 
	$\rationals[x]$-generators of $\rationals N$ normally generates a finite index subgroup of $N$. 
	Indeed, suppose that $a_1,\ldots,a_d \in \rationals N$ is a
	$\rationals[x]$-generating set for $\rationals N$. 
	Let $N_0$ the $\integers[x]$-span of $a_1,\ldots,a_d$. Then $\rationals N_0 = \rationals N'$, and so by
	Corollary~\ref{cor:two full lattices in Q^d - one contained in the other => the index is finite} we conclude that
	$N_0 \leq_f N$.
\end{proof}

%Also note that the previous lemma reduces the study of maximal subgroup growth of certain metabelian groups to the study of 
%maximal submodule growth.

We next consider groups of the form $N \rtimes \integers$, where we do \emph{not} assume that $N$ is f.g.\ as an abelian group.
\begin{proposition} %Maybe call this a theorem?
	\label{prop:growth type of (arbitrary abelian) by Z}
	Suppose we have a f.g.\ group $G = N \rtimes \integers$ with $N$ abelian. Using the notation of 
	Proposition~\ref{prop:growth type of arbitry f.g. Z[x]-module}, we have
	\[
	\maxsubgr(G) \text{ has growth type }
	\begin{cases}
	n^{d} & \text{if } d > r_{\text{max}} \\
	n^{d+1} & \text{if } d = r_{\text{max}} = r(0) \\
	n^{r_{\text{max}} + 1}/\log(n) & \text{otherwise.}
	\end{cases}
	\]
\end{proposition}
\begin{proof}
	This follows from Proposition~\ref{prop:growth type of arbitry f.g. Z[x]-module}, together with
	Lemma~\ref{lem:abelian by infinite cyclic - exact counting}.
\end{proof}

Recall $\mtriv(N)$ and $\mnontriv(N)$ from Definition~\ref{def:mtriv and mnontriv}.

\begin{lemma}
	\label{lem:abelian by free abelian - num of maximal subgroups}
	Let $G$ be a f.g.\ group with abelian $N \normal G$ such that $G/N \cong \integers^\ell$. Then
	\[
	\maxsubgr(G) \leq
	\begin{cases}
	\maxsubgr[p](\integers^\ell) +  p^\ell \cdot \mtriv[p](N) + p \cdot \mnontriv[p](N) & \text{if $n = p$ is prime} \\
	n \cdot \mnontriv(N) & \text{if $n$ is not prime},
	\end{cases}
	\]
	with equality if $G = N \rtimes \integers^\ell$.
\end{lemma} 
\begin{proof}
	Let $R = \integers[x_1,x_2,\ldots,x_\ell]$. Lemma \ref{lem:simple module - one fixed point implies trivial and has order a prime} 
	tells us that any trivial, simple quotient of the $R$-module $N$ has prime order. In other words, $\mtriv(N)$ = 0 if $n$ is not
	prime. Also, we know that $\maxsubgr(\integers^\ell) = 0$ if $n$ is not prime. So what we want to show is that for all $n$,
	\[
	\maxsubgr(G) \leq \maxsubgr(\integers^\ell) +  n^\ell \cdot \mtriv(N) + n \cdot \mnontriv(N)
	\]
	with equality if $G = N \rtimes \integers^\ell$. Both the inequality and equality follow from
	 Lemma \ref{lem:abelian by something - counting maximal subgroups by derivations}, by splitting the summation
	and applying Lemma~\ref{lem:counting derivations from free abelian groups to simple modules} to count the derivations.
\end{proof}

\begin{lemma}
	\label{lem:abelian by f.g. abelian - num of maximal subgroups}
	Let $G$ be a f.g.\ group with abelian $N \normal G$ such that $G/N \cong A$, for some abelian $A$ of torsion-free rank $\ell$. Then
	\[\tag{*}
	\maxsubgr(G) \leq  
	\maxsubgr(A) +  \lvert \Hom(A, \integers/n\integers)\rvert \cdot \mtriv(N) + n \cdot \mnontriv(N)
	\]
	which for large $n$ equals
	\[
	\maxsubgr(\integers^\ell) +  n^\ell \cdot \mtriv(N) + n \cdot \mnontriv(N).
	\]
	And if $G = N \rtimes A$, then (*) is an equality.
\end{lemma} 
\begin{proof}[Sketch of proof]
	This is extremely similar to Lemma \ref{lem:abelian by free abelian - num of maximal subgroups}. Of course, we use here
	Lemma~\ref{lem:counting derivations from f.g. abelian groups to simple modules} instead of 
	Lemma~\ref{lem:counting derivations from free abelian groups to simple modules} for the $n \cdot \mnontriv(N)$ term.
	%Maybe fill out a few more details??
\end{proof}

\begin{theorem}
	\label{thm:(f.g. abelian) by (f.g. abelian)}
	Let $G$ be a group with f.g.\ abelian normal subgroup $N$. Suppose $G/N$ is an abelian, $\ell_0$-generated group
	of torsion-free rank $\ell$. So $N$ is a $\integers[x_1,\ldots,x_{\ell_0}]$-module.
	Let $I = (x_1 - 1, x_2 - 1, \ldots, x_{\ell_0} - 1)_{\integers[x_1,\ldots,x_{\ell_0}]}$. Let $t$ be
	the torsion-free rank of (the abelian group) $N/IN$, and let $d = d_{\rationals R}(\nq)$.
	Then
	\[
	\mdeg(G) \leq \max \{\ell + t - 1, d \},
	\]
	with equality if both $G \cong N \rtimes G/N$ and $\ell \geq 1$.
	
%	 Let $N$ be a $\integers[x_1,\ldots,x_\ell]$-module which is f.g.\ as an abelian group, and suppose
%	that the module structure of $N$ came from a group $G = N \rtimes \integers^\ell$. 
%	Let $I = (x_1 - 1, x_2 - 1, \ldots, x_\ell - 1)_{\integers[x_1,\ldots,x_\ell]}$. Let $t$ be
%	the torsion-free rank (of the abelian group) $N/IN$, and let $d = d_{\rationals R}(\nq)$.
%	Then
%	\[
%	\mdeg(G) = \max \{\ell + t - 1, d \}.
%	\]
\end{theorem}
\begin{proof}
	By Lemma \ref{lem:abelian by f.g. abelian - num of maximal subgroups}, for large $n$,
	\[\tag{*1}
	 \maxsubgr(G) \leq \maxsubgr(\integers^\ell) +  n^\ell \cdot \mtriv(N) + n \cdot \mnontriv(N)
	\]
	with equality if $G \cong N \rtimes G/N$. To prove this theorem, we will just show that
	\[\tag{*2}
	  \deg(\maxsubgr(\integers^\ell) +  n^\ell \cdot \mtriv(N) + n \cdot \mnontriv(N)) \leq \max \{\ell + t - 1, d \},
	\]
	with equality if $\ell \geq 1$.
	
	First, note that $\ell +t - 1$
	equals $\deg(\maxsubgr(\integers^\ell))$ if $t = 0$ and $\deg(n^\ell \cdot \mtriv(N))$  if $t \neq 0$. 
	If we could show that 
	$d$ equals $\deg(n \cdot \mnontriv(N))$ we would be practically done, but this is not quite the case.
	
	Next, note that $\deg(\maxsubgr(\integers^\ell)) = \mdeg(\integers^\ell) = \ell -1$. 
	
	Next, note that 
	$\deg(\maxsubmod(N)) = \mmoddeg(N)$, which we would like to show is $d - 1$. (This, together with (*1), is the heart of what separates the present theorem
	from Corollary~\ref{cor:mdeg(N rtimes A) = d(QN)}.)
	Let $\tilde{N}_D$ and $d_0$ be as in
	Lemma~\ref{lem:several variables to one, with a nice direct sum too}; by this lemma, $\maxsubmod(N) = \maxsubmod(\tilde{N}_{D})$ for all large $n$.
	So $\deg(\maxsubmod(N)) = \deg(\maxsubmod(\tilde{N}_D))$. By Proposition~\ref{prop:mdeg(N) -actually tilde{m}deg(N)}, applied to the module
	$\tilde{N}_D$, we get that	$\deg(\maxsubmod(\tilde{N}_D)) = d_0 - 1$. And by Lemma~\ref{lem:size of QR generating set for QN is d_0},
	$d_0 = d_{\rationals R}(\nq)$  (which is $d$). Therefore, we have shown that $\mmoddeg(N) = \deg(\maxsubmod(N)) = d-1$.
	
	Suppose $t = 0$. Then $\mtriv(N) = 0$ for all large $n$. Hence $\mnontriv(N) = \maxsubmod(N)$ for large $n$. The previous two 
	sentences (together with (*1)) imply that 
	for large $n$,
	\[
	\maxsubgr(\integers^\ell) +  n^\ell \cdot \mtriv(N) + n \cdot \mnontriv(N) = \maxsubgr(\integers^\ell) + n \cdot \maxsubmod(N).
	\]
	Recall $t = 0$. We are done by Lemma~\ref{lem:degree of a sum is the max of the degrees} 
	since we already noted $\deg(\maxsubgr(\integers^\ell)) = \ell -1$ and since $\deg(n \cdot \maxsubmod(N))  = 1 + (d - 1) = d$.
	
	For the rest of the proof, suppose $t \neq 0$. Note that Corollary~\ref{cor:mtriv(N) is (n^t-1)/(n-1) for n prime} implies that
	$\deg(\mtriv(N)) = t-1$. Hence $\deg(n^\ell \cdot \mtriv(n)) = \ell + t - 1 $. Therefore, by 
	Lemma~\ref{lem:degree of a sum is the max of the degrees},
	\[
	 \deg(\maxsubgr(\integers^\ell) + n^\ell \cdot \mtriv(N)) = \deg(n^\ell \cdot \mtriv(N)) = \ell + t - 1.
	\]
	Therefore by Lemma~\ref{lem:degree of a sum is the max of the degrees} again,
	$\deg(\maxsubgr(\integers^\ell) +  n^\ell \cdot \mtriv(N) + n \cdot \mnontriv(N)) = $
	\[ \tag{*3}
	\begin{aligned}
								   &   \deg(n^\ell \cdot \mtriv(N) + n \cdot \mnontriv(N)) \\ 
								   &= \max\{\deg(n^\ell \cdot \mtriv(N)), \deg(n \cdot \mnontriv(N))  \}\\
								   &= \max\{\ell + t - 1, \deg(n \cdot \mnontriv(N))  \},
	\end{aligned}
	\]
	which is bounded above by $ \max\{\ell + t - 1, d  \}$ because $\mnontriv(N) \leq \maxsubmod(N)$ implies that
	$\deg(n \cdot \mnontriv(N)) \leq \deg(n \cdot \maxsubmod(N)) = 1 +(d-1) = d$. This proves (*2). So to get an equality in (*2), assume 
	$\ell \geq 1$.
	
	Because $\maxsubmod(N) = \mtriv(N) + \mnontriv(N)$, we know (by Lemma~\ref{lem:degree of a sum is the max of the degrees}) that
	\[
	\deg(\maxsubmod(N)) = \deg(\mtriv(N)) \text{\quad or}
	\]
	\[
	\deg(\maxsubmod(N)) = \deg(\mnontriv(N)).
	\]
	
	\noindent \emph{Case 1.} Assume $\deg(\maxsubmod(N)) = \deg(\mtriv(N))$. 
	Then 
	\[
	\deg(\mnontriv(N)) \leq \deg(\mtriv(N)).
	\] 
	Hence since we are now assuming $\ell \geq 1$,
	\[
	\deg(n^\ell \cdot \mtriv(N) + n \cdot \mnontriv(N)) = \deg(n^\ell \cdot \mtriv(N)),
	\] which equals $\ell + t - 1 $, which is at least $d$ since $\ell \geq 1$ and $d -1 = \deg(\maxsubmod(N)) = \deg(\mtriv(N)) = t-1$. We are done with this case
	by (*3).\\

	\noindent \emph{Case 2.} Assume $\deg(\maxsubmod(N)) = \deg(\mnontriv(N))$. 
	Then $\deg(n\cdot \mnontriv(N)) = \deg(n\cdot \maxsubmod(N))	=  1 + (d -1) = d$.
	We are done by (*3).
\end{proof}
Note: Let $m = \max \{\ell + t - 1, d \}$. A few changes to Case 1 in the proof of 
Theorem~\ref{thm:(f.g. abelian) by (f.g. abelian)} actually shows that if $G/N$ is finite abelian,
(and $G = N \rtimes G/N$) then $\mdeg(G) = m$ or $m-1$. Also, to get this, it actually turns out that (if $G/N$ is finite abelian) we do
not even need to assume $G = N \rtimes G/N$, but this latter observation requires additional work not given here.

\subsection{Nilpotent groups}
\label{sec:nilpotent groups}
This section gives a formula for calculating $\mdeg(G)$ for  all f.g.\ nilpotent groups $G$.
There are two reasons for doing this. First, at a mathematics conference at Texas~A\&M, Alex Lubotzky kindly suggested this
to the author as ``an easy exercise \footnote{More specifically, he suggested to give a formula for 
	the maximal subgroup growth of f.g. nilpotent groups.} 
	that definitely should appear in your thesis.''  Second, we would like to know how
accurate (or not) Lemma \ref{lem:abelian by something - counting maximal subgroups by derivations} is when $G$
is not a semidirect product. In Section \ref{sec:some f.g. metabelian nilpotent groups}, we apply the results of the present section to a class of examples
(certain metabelian nilpotent groups), and these groups show how inaccurate  Lemma \ref{lem:abelian by something - counting maximal subgroups by derivations}
can be when applied to groups that are not semidirect products.

Let $G$ be f.g.\ nilpotent. It is well known that a maximal subgroup of $G$ must be normal and hence have prime index.
See for example 5.2.4 in \cite{Robinson} and the comments following.\footnote{Yes, the result itself
	has the hypothesis that the group be finite, but notice that finiteness is not used in his ``(i) $\to$ (ii)'' nor in ``(ii) $\to$ (iii)''.}

\begin{definition}
	Similar to the Frattini subgroup, we define
	\[
	\Phi_p(G) = \bigcap_{M \leq_p G} M.
	\]
\end{definition}
Recall a familiar argument that shows $G/\Phi_p(G)$ to be an elementary abelian $p$-group: Let $M \leq_p G$. Then
$G/M$ is abelian. Therefore $G' \subseteq M$. Hence $G/\Phi_p(G)$ is abelian. Of course, $G/\Phi_p(G)$ has ``exponent'' $p$,
and it is finitely generated. Thus $G/\Phi_p(G)$ is in fact a finite dimensional $\finitefield$-vector space.

\begin{definition}
	We denote\footnote{See also the same notation in \cite{Lubotzky2003} (page xxii).} by $\ur_p(G)$ the dimension of  $G/\Phi_p(G)$ as an $\finitefield$-vector space.
\end{definition}

\begin{lemma}
	\label{lem:nilpotent group maximal subgroup growth}
	Let $r = \limsup_{p \to \infty} \ur_p(G)$. Then $\mdeg(G) = r - 1$, and in fact,
	\[
	\maxsubgr[p](G) \leq \frac{p^r - 1}{p - 1} \text{\hspace{.2in}  for all large } p,
	\]
	with equality for infinitely many $p$.
\end{lemma}
\begin{proof}
	Because $G/\Phi_p(G)$ is an $\finitefield$-vector space of dimension $\ur_p(G)$, we know that it (and hence $G$)
	has $\frac{p^{\ur_p(G)} - 1}{p - 1}$ subgroups of index $p$.  
\end{proof}

\subsection{Some f.g.\ metabelian nilpotent groups}
\label{sec:some f.g. metabelian nilpotent groups}
We will next form a class of examples of f.g.\ metabelian nilpotent groups $G_f$ each of which has a 
normal subgroup $N$ such that both $N$ and $G_f/N$ are free abelian.

Fix $\ell \geq 2$, and let $k = \binom{\ell}{2}$. Write $\integers^k$ multiplicatively having generating 
set $\{y_1, \ldots, y_k\}$. Choose a function $f: \{(i,j) | 1 \leq i < j \leq \ell \} \longrightarrow \integers^k$.
Let $[k] = \{1, 2, \ldots, k\}$ and similarly $[\ell] = \{1, 2, \ldots, \ell\}$.
Form the group $G_f$, a presentation of which has generating set
$\{x_1, \ldots x_\ell, y_1, \ldots y_k \}$ and relations $[x_i, x_j] = f(i,j)$ for $1 \leq i < j \leq \ell$, 
$ [y_i, y_j] = 1$  for all $i, j \in [k]$, $[x_i, y_j] = 1$  for all $i \in [\ell], j \in [k]$.

%The following would be nice if I could write it in multiple lines.
%\[
% G_f = \langle x_1, \ldots x_\ell, y_1, \ldots y_k | [y_a, y_b] = 1 \forall a, b \in [k], [x_i, x_j] = 1 \forall 1 \leq i < j \leq \ell 
%\rangle
%\]

So $G_f$ has the subgroup $A  = \langle y_1, \ldots, y_k \rangle = \integers^k$ with $A \subseteq Z(G_f)$, 
and also $G_f/A \cong \integers^\ell$. Thus $G_f$ is nilpotent and metabelian.

Form the (central) subgroup
\[
N = \langle f(i,j) \rangle_{1 \leq i < j \leq \ell}.
\]
Of course, $N \leq A$. Since we are using multiplicative notation, for a given prime $p$, modding out by $p$ gives
$N/N^p$, an $\finitefield$-vector space. 

\begin{lemma}
	\label{lem:ur_p(G_f) = ell + k - dim(N/N^p)}
	Fix a prime $p$. Then
	\[
	\ur_p(G_f) = \ell + k - \dim_{\mathbb{F}_p}(N/N^p).
	\]
\end{lemma}
\begin{proof}[Sketch of proof]
	Forming $G_f/\Phi_p(G_f)$ is straightforward because $N \subseteq \Phi_p(G_f)$ and also $G_f^p \subseteq \Phi_p(G_f)$.
	%$\finitefield^{\ell + k}$ modded out by all then $p$ powers (i.e.\ multiples) of $N$.
\end{proof}

Since $N$ is a subgroup of a free abelian group of rank $k$ (namely $A$), we may view $N$ as a subset of
$\rationals^k = \rationals \tensor_\integers A$. The subspace of $\rationals^k$ spanned by $N$ is 
$\rationals \tensor N$. The following is clear:

\begin{lemma}
	\label{lem:dim(N/N^p) over F_p = dim(Q tensor N) for almost all primes}
	For almost all primes $p$
	\[
	\dim_{\mathbb{F}_p}(N/N^p) = \dim_\rationals(\rationals \tensor N).
	\]
\end{lemma}

Recall that $k = \binom{\ell}{2}$. Note that by choosing $f$ appropriately, we may pick
$\dim_\rationals(\rationals \tensor N)$ to be any number in $\{\binom{\ell}{2}, \binom{\ell}{2}  - 1, \ldots, 1, 0 \}$.
So by using
Lemmas \ref{lem:nilpotent group maximal subgroup growth}, \ref{lem:ur_p(G_f) = ell + k - dim(N/N^p)}, and
\ref{lem:dim(N/N^p) over F_p = dim(Q tensor N) for almost all primes}, we can make
$\mdeg(G_f)$ any number in $\{\ell - 1, \ell, \ell + 1, \ldots, \ell +\binom{\ell}{2} -1 \}$.  
And this tells us how inaccurate Lemma \ref{lem:abelian by something - counting maximal subgroups by derivations} is in general 
because that lemma in this situation ends up saying $\maxsubgr(G_f) \leq \maxsubgr(\integers^{\ell + k})$,
but $\mdeg(\integers^{\ell + k}) = \ell + k - 1 = \ell +  \binom{\ell}{2} -1$; specifically, we have groups $G_f$ such that
$\mdeg(G_f) = \ell - 1$ but for which Lemma~\ref{lem:abelian by something - counting maximal subgroups by derivations} only implies that
$\mdeg(G_f) \leq \ell +  \binom{\ell}{2} -1$.

\subsection{A concrete example: $ \integers^3 \rtimes \integers/3\integers$}
In this section, we calculate $\maxsubgr(G)$ exactly for $G = \integers \wr \integers/3\integers = \integers^3 \rtimes \integers/3\integers$. 

Let $R = \integers[x]/(x^3 - 1)$. The $R$-module structure of $\integers^3 \normal G$ is $R$. 
So, we need to calculate $\maxsubmod(R)$ for all $n$. Of course, 
$x^3 - 1 = (x-1)(x^2 + x + 1)$, and so our goal is to factor $$f(x)  = x^2 + x + 1$$ mod $p$ for all primes $p$. If $p = 2$,
then we easily see that $f$ is irreducible mod 2 because it has no roots mod 2. For $p = 3$, we see that
$x^2 + x  + 1 \equiv (x-1)^2$. 

So far, we've shown that $\maxsubmod[2](R)  = 1$, $\maxsubmod[2^2](R)  = 1$, $\maxsubmod[3](R)  = 1$, and
that $R$ has no other ideal of index a power of 2 or 3. 

\begin{lemma}
	\label{lem:f is reducible implies - f is irreducible implies}
	Let $p \neq 3$.  If $f$ is irreducible mod $p$, then $\maxsubmod[p^2](R) = 1$ and $\maxsubmod[p](R) = 1$.
	If $f$ is reducible mod $p$, then $\maxsubmod[p](R) = 3$ and $\maxsubmod[p^k](R) = 0$ for $k \geq 2$.
\end{lemma}
\begin{proof}[Sketch of proof:]
	We have already shown this for $p = 2$. 
	Of course, the factor $x - 1$ in $x^3 - 1$ is why $\maxsubmod[p](R) \geq 1$ for all primes $p$.
	The only other thing about this lemma that may need comment/proof is why $f$ factors into distinct factors mod $p$ if it 
	is reducible; see the paragraph after Lemma \ref{lem:quadratic reciprocity used to factor x^2 + x + 1}.
\end{proof}

It is well known how to factor $f(x) = x^2 + x + 1$ mod $p$ for all primes $p > 3$, but
we show the computation in detail. We will use the notation $\Legendre{a}{p}$ for the Legendre symbol.

\begin{lemma}
	\label{lem:quadratic reciprocity used to factor x^2 + x + 1}
	Let $p \neq 3$. Then the above $f$ is reducible mod $p$ if and only if $p \equiv 1$  $\mod 3$.
\end{lemma}
\begin{proof}
	Recall/notice that the quadratic formula works in $\mathbb{F}[x]$ for any field $\mathbb{F}$ because completing
	the square works. So $f$ is reducible in $\finitefield[x]$ if and only if $-3$ is the square of some number in $\finitefield$.
	
	The lemma is true for $p = 2$. Suppose $p > 3$. Since $\Legendre{\cdot}{p}$ is a homomorphism from $\finitefield^{\neq 0}$ to 
	$\{1, -1\}$, we get $\Legendre{-3}{p} = \Legendre{-1}{p}\Legendre{3}{p}$.
	
	We finish by using the law of quadratic reciprocity:
	$\Legendre{-1}{p} = (-1)^{(p-1)/2}$ and $\Legendre{3}{p} = (-1)^{(3-1)(p-1)/4} \Legendre{p}{3}$. Therefore,
	$\Legendre{-3}{p} = ((-1)^{(p-1)/2})^2 \Legendre{p}{3} = \Legendre{p}{3}$. So we see that $-3$ is a square mod $p$
	if and only if $p$ is a square mod 3. But 1 is the only (nonzero) square in $\mathbb{F}_3$. So $p$ is a square mod 3 if and only if
	$p \equiv 1 \mod 3$. Just recall the second sentence of the first paragraph.
\end{proof}
Recalling Lemma \ref{lem:f is reducible implies - f is irreducible implies},
we see now why if $f$ is reducible mod $p$, for $p \neq 3$, then $f$ factors into a product of two distinct terms; this 
is because of the quadratic formula and that $-3 \not\equiv 0 \mod p$ and because $(-1)^2 \not\equiv -3$ implies
$(1 \pm \sqrt{-3})/2 \not\equiv 1 \mod p$.

We can now combine Lemmas \ref{lem:f is reducible implies - f is irreducible implies} and 
\ref{lem:quadratic reciprocity used to factor x^2 + x + 1} to get  that for $p \neq 3$, 
\[\begin{aligned}
\maxsubmod[p](R)     &= 
\begin{cases}
3 & \text{ if } p \equiv 1 \mod 3 \\
1 & \text{ if } p \not\equiv 1 \mod 3
\end{cases}  \\
\maxsubmod[p^2](R) &= 
\begin{cases}
0 & \text{ if } p \equiv 1 \mod 3 \\
1 & \text{ if } p \not\equiv 1 \mod 3
\end{cases}  \\
\end{aligned}
\]
We have already stated that $\maxsubmod[3](R) = 1$. For all other $n > 1$ not listed, we have $\maxsubmod(R) = 0$.

Recall that $G = \integers \wr \integers/3\integers$. We use the first part of 
Lemma \ref{lem:abelian by f.g. abelian - num of maximal subgroups} to calculate $\maxsubgr[3](G)$: in the 
lemma's notation, $A = \integers/3\integers$, and the $R$-module $N$ is itself $R$. So 
\[\begin{aligned}
\maxsubgr[3](G) &= \maxsubgr[3](\integers/3\integers) 
+  \lvert\Hom(\integers/3\integers, \integers/3\integers)\rvert \cdot \mtriv[3](R) + 3 \cdot \mnontriv[3](R)\\
&= 1 + 3(1) + 3(0) = 4. 
\end{aligned}
\]
%This agrees with the number of subgroups of \integers/3  \wr \integers/3 of ORDER 27 equaling 4, as stated on
% http://groupprops.subwiki.org/wiki/Subgroup_structure_of_groups_of_order_81
Also, $\maxsubgr[3^k](G)  = 0$ for $k \geq 2$ because (as stated right before 
Lemma \ref{lem:f is reducible implies - f is irreducible implies}) $R$ has no maximal 
ideals of index $3^k$ (for such $k$). 

We next apply Lemma \ref{lem:abelian by f.g. abelian - num of maximal subgroups} again: Let $n$ be a power
of a prime $p \neq 3$. Then $\maxsubgr(\integers/3\integers) =0$ and $\lvert \Hom(\integers/3\integers, \integers/n\integers)\rvert = 1.$
Thus Lemma \ref{lem:abelian by f.g. abelian - num of maximal subgroups} simplifies to the following:
\[
  \maxsubgr(G) = \mtriv(R) + n \cdot \mnontriv(R).
\]

For all $n$, if $n$ is not prime, then $\mnontriv(R) = \maxsubmod(R)$ and $\mtriv(R) = 0$. Also, for all primes $p$, $\mtriv[p](R) = 1$, and
\[
\mnontriv[p](R)    = 
\begin{cases}
2 & \text{ if } p \equiv 1 \mod 3 \\
0 & \text{ if } p \not\equiv 1 \mod 3.
\end{cases} 
\]

Combining our work so far gives the following:
\begin{proposition}
  Let $G = \integers \wr \integers/3\integers$. Then $\maxsubgr[3](G) = 4$. Let $p \neq 3$ be prime. Then 
  \[\begin{aligned}
  \maxsubgr[p](G)     &= 
  \begin{cases}
  1 + 2p & \text{ if } p \equiv 1 \mod 3 \\
  1 & \text{ if } p \not\equiv 1 \mod 3
  \end{cases}  \\
  \maxsubgr[p^2](G) &= 
  \begin{cases}
  0 & \text{ if } p \equiv 1 \mod 3 \\
  p^2 & \text{ if } p \not\equiv 1 \mod 3.
  \end{cases}  \\
  \end{aligned}
  \] 
  For all other $n$, we get $\maxsubgr(G) = 0$. So $\maxsubgr(G) \leq 2n + 1$ 
  for all $n$ with equality for infinitely many $n$. In particular, $\mdeg(G) = 1$.
\end{proposition}

\section*{Acknowledgments}
I would like to thank Marcin Mazur, who was my advisor while at Binghamton University. I also want to thank
Alex Lubotzky for a couple helpful conversations.

\bibliography{my_bibliography}{}
\bibliographystyle{plain}

\textsc{Department of Mathematics and Computer Science, Colorado College,
	Colorado Springs, Colorado 80903}\par\nopagebreak
\textit{email address}: \texttt{akelley@coloradocollege.edu}\footnote{The author is a Visiting Assistant Professor of Mathematics at Colorado College.
	A possibly more permanent email address is \texttt{akelley2500@gmail.com} }

\end{document}